\documentclass[11pt,reqno]{amsart}
\usepackage{mathrsfs,amsmath,graphicx,color}
\usepackage{amsmath,mathrsfs,amsfonts,amssymb}
\usepackage{cases}
\usepackage{mathtools}
\usepackage{stmaryrd}
\usepackage{color}
\usepackage{hyperref}
\numberwithin{equation}{section}

\newtheorem{theorem}{{\bf Theorem}}[section]

\newtheorem{lemma}[theorem]{{\bf Lemma}}
\newtheorem{prop}[theorem]{{\bf Proposition}}
\newtheorem{remark}[theorem]{{\bf Remark}}

\newcommand{\ve}{\varepsilon}
\newcommand{\p}{\partial}
\renewcommand{\O}{\Omega}
\renewcommand{\div}{\operatorname{div}}

\newcommand{\tr}{\operatorname{tr}}

\newcommand{\N}{\mathcal{N}}

\newcommand{\Q}{\mathcal{Q}}
\newcommand{\HH}{\mathrm{H}}

\newcommand{\1}{\mathbf{1}}

\newcommand{\R}{\mathbb{R}}
\newcommand{\nn}{\mathrm{n}}
\newcommand{\uu}{\mathrm{u}}

\renewcommand{\tilde}{\widetilde}

\def\({\left(}
\def\){\right)}
\def\l|{\left|}
\def\r|{\right|}
 \def\[{\begin{equation}}
 \def\]{\end{equation}}

\newcommand{\BS}{{\mathbb{S}^2}}

\begin{document}

\title{Nematic-Isotropic phase transition in Liquid crystals:\\ a variational derivation of   effective geometric motions}

\author{Tim Laux}
\address{Hausdorff Center for Mathematics, University of Bonn, Villa Maria, Endenicher Allee 62, D-53115 Bonn, Germany}
\email{tim.laux@hcm.uni-bonn.de}
\author{Yuning Liu }
\address{NYU Shanghai, 1555 Century Avenue, Shanghai 200122, China,
and NYU-ECNU Institute of Mathematical Sciences at NYU Shanghai, 3663 Zhongshan Road North, Shanghai, 200062, China}
\email{yl67@nyu.edu}

\begin{abstract}
In this work, we study the nematic-isotropic phase transition based on  the dynamics of the  Landau--De Gennes theory of liquid crystals.  At the critical temperature,  the Landau--De Gennes bulk potential favors the  isotropic phase and nematic phase  equally. When the elastic coefficient is much smaller  than that of the bulk potential, a scaling limit can be derived  by formal asymptotic expansions: the solution gradient concentrates on a closed  surface evolving by mean curvature flow. Moreover, on one side of the surface  the solution tends to the nematic phase which  is governed by the harmonic map heat flow into the sphere while on the other side,  it tends to the isotropic phase. To rigorously  justify such a scaling limit,  we prove a   convergence result by combining weak convergence methods and the modulated energy method. Our proof applies  as long as   the limiting mean curvature flow  remains smooth.
\end{abstract}

 \maketitle





\section{Introduction}

Nematic liquid crystals react to shear stress like a conventional liquid while the molecules are oriented in a crystal-like way.
One of the successful continuum theories modeling nematic liquid crystals is the $Q$-tensor theory,   also referred to as Landau--De Gennes theory, which  uses a $3\times 3$ traceless and symmetric matrix-valued function $Q(x)$ as order parameter to characterize the orientation of molecules near a material point $x$ (cf. \cite{DeGennesProst1995}).
The matrix $Q$, also called $Q$-tensor, can be interpreted as the second moment of a number density function
\begin{equation}
Q(x)=\int_{\mathbb{S}^2} (p\otimes p-\tfrac13I_3)f(x,p)\, dp,
\end{equation}
where $f(x, p)$   corresponds to the number density of liquid crystal molecules which orient along the direction $p\in\BS$ near the material point $x$ (cf. \cite{maiersaupe}).
The configuration space of the $Q$-tensor is the $5$-dimensional linear space \begin{equation}\label{Q-space}
  \Q=\{Q\in \R^{3\times 3}\mid Q =Q^T,~\tr Q =0\}.
\end{equation}
  By elementary linear algebra, each such $Q$ can be written as
  \begin{equation}
  Q=s\left(\mathrm{u} \otimes \mathrm{u}-\frac{1}{3} I_3\right)+t\left(\mathrm{v} \otimes \mathrm{v}-\frac{1}{3} I_3\right),
    \end{equation}
for some $s,t\in\R$ and $\mathrm{u},\mathrm{v}\in \BS$ which are perpendicular.
In the physics literature, for instance De Gennes--Prost \cite{DeGennesProst1995}, such a representation is called the  biaxial nematic configurations, cf. \cite{MZ}. In case  $Q$ has  repeated eigenvalues, it is   called uniaxial. These $Q$'s form a  $3$-dimensional   manifold  in $\Q$, denoted by
\begin{equation}\label{uniaxial}
\mathcal{U}:=\left\{Q\in \Q~\Big|~ Q=s\left(\mathrm{u} \otimes \mathrm{u}-\frac{1}{3} I_3\right)\quad \text{for some}~ s \in \mathbb{R}~\text{and}~ \mathrm{u}  \in \BS\right\},
\end{equation}
with a conical  singularity at $s=0$.
Here the parameter $s$ is called the degree of orientation. To study  static configurations of the liquid crystal material in a physical domain $\Omega$, a natural approach is to  consider the  Ginzburg--Landau type
energy
\begin{equation}\label{GL energy}
E_\ve(Q)=\int_{\O} \(\frac\ve2 |\nabla Q|^2+\frac{1}\ve F(Q)  \)\, dx,
\end{equation}
where $\O\subset \R^d$ is a bounded domain with smooth boundary, $|\nabla Q|=\sqrt{\sum_{ijk}|\p_k Q_{ij}|^2}$,  and  $F(Q)$ is the bulk energy density
\begin{equation}\label{bulkED}
F(Q)=
\frac{a}{2}\mathrm{tr}(Q^2)-\frac{b}{3}\tr(Q^3)+\frac{c}{4}\(\tr(Q^2)\)^2.
\end{equation}
Here  $a,b,c \in \mathbb{R}^+$ are material and temperature dependent constants, and $\ve$ denotes the relative intensity of elastic and  bulk energy, which is  usually quite small.
It can be proved  that all critical points of $F(Q)$ are uniaxial \eqref{uniaxial}, (cf. \cite{MZ}), and thus
 \begin{equation}\label{uni bulk}
    F (Q)=\frac{s^2}{27}(9a-2bs+3cs^2)=:f(s),~\text{if $Q$ is uniaxial \eqref{uniaxial}}.
 \end{equation}
Moreover,  $F(Q)$ has two families of stable local minimizers corresponding to the following choices of $s=s_\pm$:
  \begin{equation}\label{splus}
 s_-=0,\qquad s_+=\frac{b+\sqrt{b^2-24ac}}{4c}.
  \end{equation}
  In this work we shall consider the bistable case when
\begin{equation}\label{critical coeff}
b^2=27ac,\quad\text{and}~  a,c>0.
\end{equation}
By rescaling, one can choose $a=3,b=9,c=1$.
From the physics view point, such  choices of the coefficients  correspond to the critical temperature at which the system favors the nematic phase  and the isotropic phase equally \cite[Section 2.3]{DeGennesProst1995}.
Analytically,  it can be shown  that, in this case,  the  two families of minimizers corresponding to \eqref{splus}  are  the only global minimizers of $F(Q)$:
\begin{equation}\label{lemma1}
F(Q)\geq 0~ \text{and  the equality holds  if and only if}~ Q\in \{0\}\cup \N,
\end{equation}
 where
   \begin{equation}\label{manifold}
  \N:=\left\{Q\in\Q\mid Q=s_+\(\mathrm{u}\otimes \mathrm{u}-\frac 13I_3\)~\text{for some}~\mathrm{u}\in\BS\right\},\qquad \text{with}~s_+=\sqrt{\frac {3a}c}.
\end{equation}

  At this point we digress to mention that  the Landau--De Gennes   model  \eqref{GL energy} is closely related to  Ericksen's model, where the energy is
  \begin{equation}\label{ericksen}
e_{E}(s,\mathrm{u}):= \int_{\Omega}\left(\kappa|\nabla s|^{2}+s^{2}|\nabla \mathrm{u}|^{2}+\psi(s)\right) \,  d x.
 \end{equation}
This model  was  introduced  by Ericksen \cite{EricksenARMA1990}  for the purpose of studying line defects. It  can be formally  obtained by plugging the uniaxial Ansatz \eqref{uniaxial} into \eqref{GL energy}.
   In contrast to \eqref{GL energy} which uses $Q\in \Q$ as order parameter, Ericksen's model uses $(s,\mathrm{u})\in\R\times \BS$ and is very useful  to   describe liquid crystal defects.
The analysis of this  model is very challenging, mainly due to the reason that  the geometry  of the uniaxial configuration \eqref{uniaxial}
 corresponds to   a double-cone, and the energy \eqref{ericksen} is highly degenerate when $s=0$.  The analytical  aspects of such a model have been  investigated by many authors, for instance, by Lin \cite{Lin1991}, Hardt--Lin--Poon \cite{MR1294333}, Bedford \cite{MR3437868}, Alper--Hardt--Lin \cite{MR3689151}, and Alper \cite{MR3749263}.


To model nematic-isotropic phase transitions in the framework of   Landau--De Gennes theory,  we shall investigate the small-$\ve$ limit  of the natural   gradient flow dynamics  of \eqref{GL energy} with initial data undergoing  a sharp transition near a smooth interface. To be more precise, we consider the system
\begin{subequations}\label{Ginzburg-Landau sys}
\begin{align}
\partial_{t} Q_\ve&=\Delta Q_\ve - \frac 1{\ve^2}\nabla_q F(Q_\ve),\,\,\,\text{in}~ \Omega\times (0,T),\label{Ginzburg-Landau}\\
Q_\ve(x,0)&=Q_\ve^{in}(x),\,\,\,\,\,\,\,\,\,\,\qquad\qquad\qquad~\text{in}~\Omega,\\
Q_\ve(x,t)&=0,\,\,\,\,\,\,\,\,\,\,\qquad\qquad\qquad\qquad~\, ~ \text{on}~\p\O\times (0,T),\label{bc of omega}
\end{align}
 \end{subequations}
where $\nabla_q F(Q)$ is the variation of $F(Q)$ in space $\Q$:
\begin{equation}\label{pF}
(\nabla_q F(Q))_{ij}=a Q_{i j}-b \sum_{k=1}^3Q_{i k} Q_{k j}+c|Q|^{2} Q_{i j}+\frac{b}{3} |Q|^{2} \delta_{i j}.
\end{equation}
The system \eqref{Ginzburg-Landau} is the $L^2$-gradient flow of energy \eqref{GL energy} on the slow time  scale $\ve$. 

Our main result, Theorem \ref{main thm}, states that  starting from initial conditions with a reasonable nematic-isotropic phase transition from a nematic region $\O^+(0)$ into an isotropic region $\O^-(0)$, before the occurrence of topological changes, the solution $Q_\ve$ of \eqref{Ginzburg-Landau sys} converges to the isotropic phase $Q\equiv 0$ in $\O^-(t)$ and to a field $Q\in\N$ taking values in the nematic phase in $\O^+(t)$, where the interface between $\O^+(t)$ and $\O^-(t)$  moves by mean curvature flow.
Furthermore, we show that the limit  $Q$ is a harmonic map heat flow from $\O^+(t)$ into the closed manifold $\N$. Finally, if the region $\O^+(t)$ is simply-connected, there exists a director field $\uu$ such that $Q=s_+(\uu\otimes \uu-\frac 13 I_3)$, $\uu$ is a harmonic map heat flow from $\O^+(t)$ into $\BS$, and satisfies homogenous Neumann boundary conditions on the evolving boundary $\p \O^+(t)$. 

The proof consists of two key steps: (i) an adaptation of the modulated energy inequality in \cite{fischer2020convergence} to the vector-valued case to {control the leading-order energy contribution, which is of order $O(1)$ and comes  from the phase transition across $\p \O^+(t)$.} (ii) A version of Chen--Shatah's wedge-product trick  in the sense that \eqref{Ginzburg-Landau sys} implies
\begin{equation}\label{wedge Q equation}
[\p_t Q_\ve,Q_\ve]=\nabla\cdot [\nabla Q_\ve,Q_\ve]
\end{equation}
where $[\cdot,\cdot]$ denotes the commutator.

In (i) we basically follow \cite{fischer2020convergence} but need to carefully regularize the metric $d^F$ on $Q$ induced by the conformal structure $F(Q)$ in order to exploit the fine properties of its derivative $\nabla_q d^F_\ve$.  In particular, we will use     the crucial commutator relation $\left[\nabla_q d^F_\ve(Q_\ve), Q_\ve\right]=0$ for a.e.\ $(x,t)$. This  seems to lie beyond the realm of generalized chain rules as in \cite{MR969514}, which was employed in the work of Simon and one of the  authors in \cite{MR3847750}. Regarding (ii), we emphasize that the Neumann boundary condition along the free boundary $\p\O^+(t)$ can be naturally encoded in the distributional formulation of \eqref{wedge Q equation} by enlarging the space of test functions. This however, requires uniform $L^2$-estimates on the commutators $[\p_t Q_\ve,Q_\ve]$ and $[\nabla Q_\ve,Q_\ve]$, which are one order of $\ve$ better than the a priori estimates suggest. We show that these estimates are guaranteed by our bounds on the modulated energy.

 \section{Main results}\label{sec Main thm}
To state the main result of this work, 
we assume
\begin{equation}\label{interface}
 I=\bigcup_{t\in [0,T]}\(I_t \times \{t\}\)~\text{is a smoothly evolving closed surface in}~\O,
\end{equation}
 starting from a closed smooth surface $I_0\subset  \O$. Let $\Omega^+(t)$ be the domain enclosed by $I_t$, and   $d(x,I_t)$ be   the signed-distance  from $x$ to $I_t$ which takes positive  values in  $\Omega^+(t)$, and negative values  in $\O^-(t)=\O\backslash \overline{\Omega^+(t)}$, where
 \begin{equation}\label{def:omegapm}
\Omega^{\pm}(t):= \{x\in\Omega\mid d (x,I_t)\gtrless0\}.
\end{equation}
Moreover, for each $T>0$ we shall denote the `distorted' parabolic cylinder  by
\begin{equation}\label{distorted cylinder}
\Omega^\pm_T:=\bigcup_{t\in (0,T)}\(\Omega^\pm(t)\times \{t\}\).
\end{equation}
For   $\delta>0$,  the $\delta$-neighborhood of $I_t$ is denoted by
\begin{equation}
I_t(\delta):= \{x\in\Omega:  | d (x,I_t)|<\delta\}.
\end{equation}
So there exists a sufficiently small number
$\delta_I\in (0,1)$   such that the nearest point projection $P_{I}(\cdot,t): I_t(\delta_I) \rightarrow I_t$ is smooth for any $t\in [0,T]$, and   the interface \eqref{interface} stays at least $\delta_I$ distance away  from the boundary of the domain $\p\O$.

To introduce  the modulated energy for \eqref{Ginzburg-Landau sys},  we extend the inner normal vector field $\operatorname{n}_{I}$ of  $I_t$ to a neighborhood of it by
\begin{equation}\label{def:xi}
\xi(x,t):=\eta\left(d(x, I_t)\right) \operatorname{n}_{I}\left(P_{I}(x,t),t\right)
\end{equation}
  where  $\eta$ is a  cutoff function satisfying
\begin{align} \label{cutoff func}
&\eta~\text{is even in}~\R~\text{and decreases in}~[0,\infty);\nonumber \\
& \eta(z) =  1-    z^{2}, \text { for }|z| \leq  \delta_I/2;\quad \eta(z) =0  \text { for }|z| \geq  \delta_I.
  \end{align} 
  Following \cite{MR3353807,fischer2020convergence}, we define  the modulated energy by
\begin{align}\label{entropy}
E_\ve [Q_\ve | I](t):= & \int_\O \(\frac{\varepsilon}{2}\left|\nabla Q_\ve(\cdot,t)\right|^2+\frac{1}{\varepsilon} {F_\ve(Q_\ve(\cdot,t))}- \xi\cdot\nabla \psi_\ve(\cdot,t) \)\, dx,
\end{align}
where
\begin{subequations}
\begin{align}
F_\ve(q)& {:=F(q)+\ve^{K-1}~\text{with}~K=4,}\label{new bulk}\\
\psi_\ve(x,t)&:= d^F_\ve\circ Q_\ve(x,t),\quad\text{and}~ d^F_\ve(q):=(\phi_{\ve} *d^F)(q),~\forall q\in\Q,\label{psi}
\end{align}
\end{subequations}
 and the convolution is understood in the space  $\Q\simeq \R^5$. 
 Moreover,  we set
\begin{equation}\label{psi convolu}
\phi_{\ve}(q):=\ve^{-5K}\phi\(  \ve^{-K}q\),
\end{equation}  a  family of mollifiers  in the $5$-dimensional configuration space \eqref{Q-space}.
Here  $\phi$ is  smooth, non-negative, having support  in $B_1^\Q$ (the unit ball in $\Q$),  and  isotropic, i.e. for any orthogonal matrix $R\in O(3)$ and any $q\in \Q$ it holds $\phi(R^TqR)=\phi(q)$.   The function   
$d^F$ in \eqref{psi} is the quasi-distance function
 \begin{equation}\label{quasidistance}
d^F(q):=\inf\left\{\int_0^1\sqrt{2 F(\gamma(t))}|\gamma'(t)|\, dt \Big|  \gamma\in C^{0,1}([0,1];\Q),\gamma(0)\in \N,\gamma(1)=q\right\},
\end{equation}
which    was introduced by  Sternberg  \cite{MR930124} and independently by Tartar-Fonseca \cite{MR985992} for  the study of the  singular perturbation problem.
Some properties of $d^F$ are stated in Lemma \ref{lem2} below, and interested readers can find the proof in  \cite{MR930124,Lin2012a}.  One can refer to Section \ref{sec pre} for more details of these functions.  
Throughout, we will assume an $L^\infty$-bound of $Q_\ve$, i.e.
\[\|Q_\ve\|_{L^\infty(\Omega\times(0,T))}\leq c_0\label{basic constant1}\]
for some fixed constant $c_0$. Such an estimate can be obtained by assuming an uniform $L^\infty$-bound of the initial data $Q_\ve^{in}$ and then applying maximum principle to \eqref{Ginzburg-Landau}, see Lemma \ref{L infinity bound} in the sequel.   Note that the choice $K=4$ in \eqref{new bulk} is due to a technical reason, and is  used in the proof of  Lemma \ref{basic control est}.

The main result of this work is the following:
\begin{theorem}\label{main thm}
Assume the surface $I_t$   \eqref{interface} evolves by    mean curvature flow and encloses a simply-connected domain $\Omega^+(t)$. If  the initial datum $Q_\ve^{in}$ of  \eqref{Ginzburg-Landau sys} is well-prepared in the sense that  \begin{equation}\label{initial}
 {E_\ve [Q_\ve| I](0)\leq c_1\ve,}
\end{equation}
for some constant $c_1$ that does not depend on $\ve$,  then   {for some $\ve_k\downarrow 0$ as $k\uparrow +\infty$,} \begin{equation}\label{strong global of Q}
  Q_{\ve_k}\xrightarrow{k\to\infty }   Q=s_\pm\left(\mathrm{u}(x,t) \otimes \mathrm{u}(x,t)-\tfrac{1}{3} I_3\right),~\text{strongly in}~  C([0,T];L^2_{loc}(\Omega^\pm(t))),
\end{equation}
  where $s_\pm$ are given by \eqref{splus}  and
\begin{equation}\label{reg limit}
\mathrm{u}\in H^1( \Omega^+_T;\BS).
\end{equation}
Moreover,  $\uu$ is a harmonic map heat flow into $\BS$ with homogenous Neumann boundary conditions in the sense that 
\begin{equation}\label{weak harmonic}
\int_{\Omega^+_T} \p_t \mathrm{u}\wedge \mathrm{u}\cdot \varphi\, dxdt=-\sum_{j=1}^d\int_{\Omega^+_T} \p_j  \mathrm{u}  \wedge\mathrm{u} \cdot  \p_j \varphi\, dxdt\qquad \forall   \varphi\in C^1(\overline{\Omega}\times [0,T];\R^3),
\end{equation}
where $\wedge$ is the wedge product in $\R^3$.
\end{theorem}

\begin{remark}
Note that \eqref{weak harmonic} encodes both the harmonic map heat flow into $\BS$ and the boundary conditions. Indeed, if  $(\p_t\mathrm{u},\nabla^2 \mathrm{u})$ is  continuous up to the boundary of $\Omega^+(t)$,  then the weak formulation \eqref{weak harmonic} implies that $\mathrm{u}$ is a harmonic heat flow into $\BS$ with Neumann boundary conditions on   $I_t$:
\begin{equation}
\p_t \mathrm{u}=\Delta \mathrm{u}+ |\nabla \mathrm{u}|^2\mathrm{u}~\text{in}~\O_T^+,\qquad \p_{\mathrm{n}_I} \mathrm{u}=0~\text{on}~\bigcup_{t\in (0,T)}\(I_t\times \{t\}\).
\end{equation}
 If $\Omega^+(t)$ is multi-connected, for instance when $\Omega^+(t)$ is the region outside $I_t$, then a well-known  orientability  issue arises and the conclusion  \eqref{reg limit} usually only holds away from defects. See the work of  Bedford \cite{MR3437868} for more discussions of such issues.
\end{remark}

Theorem \ref{main thm} solves a special case of the Keller--Rubinstein--Sternberg problem \cite{MR978829} using the energy method. A similar result has been  established previously  by Fei et al.~\cite{fei2015dynamics,MR4059996} using matched asymptotic expansions and spectral gap estimates. Our approach has the superiority that it allows more flexible initial data, as indicated by Proposition \ref{prop initial data} below. The general case of the Keller--Rubinstein--Sternberg problem  is fairly sophisticated and  remains open. We refer the interested  readers to   a recent work of Lin--Wang \cite{MR4002307} for the well-posedness of the limiting system. On the other hand, the static problem has been investigated  by Lin et al.~\cite{Lin2012a}.   It is worthy to   mention that recently   Golovaty et al.~\cite{MR3910590,MR4076075} studied  a model problem based on highly disparate elastic constants. Most
Recently, Lin--Wang \cite{lin2020isotropic} studied  isotropic-nematic transitions based  on an anisotropic Ericksen's model.

\medskip
Now we turn to  the construction of   initial data $Q_\ve^{in}$ satisfying \eqref{initial}.
Let $I_0\subset \O$ be a smooth closed surface and let $I_0(\delta_0)$ be a neighborhood   in which  the signed distance function $d(x,I_0)$ is smooth.
Let $\zeta(z)$ be a cut-off function such that
\begin{equation}\label{cut-off zeta}
  \zeta(z)=0~\text{for}~|z|\geq 1,~\text{and}~\zeta(z)=1~\text{for}~|z|\leq 1/2.
\end{equation} Then we define
\begin{equation}\label{cut-off initial}
\tilde{S}_\ve(x):= \zeta\(\frac {d(x,I_0)}{\delta_0}\)S\(\frac{d(x,I_0)}\ve\)+\(1-\zeta\(\frac {d(x,I_0)}{\delta_0}\)\)s_+\1_{\Omega^+(0)},
\end{equation}
where  $S(z)$ is given by the optimal profile
\begin{equation}\label{degree of orientation}
S(z):=\frac{s_+}{2}\left(1+\tanh \left(\frac{\sqrt{a}}{2} z\right)\right),\qquad z\in \R.
\end{equation}
\begin{prop}\label{prop initial data}
For  every  $\mathrm{u}^{in}\in H^1(\O;\BS)$,  the initial datum defined by
\begin{equation}\label{sharp initial}
Q^{in}_\ve(x):=\tilde{S}_\ve(x)\(\mathrm{u}^{in}(x)\otimes \mathrm{u}^{in}(x)-\frac 13 I_3\)
\end{equation}
satisfies $Q_\ve^{in}\in H^1(\O;\Q)\cap L^\infty(\O;\Q)$ and
\begin{align}\label{transition initial data}
Q_\ve^{in}(x)=\left\{
\begin{array}{rl}
s_+(\mathrm{u}^{in}\otimes \mathrm{u}^{in}-\frac 13 I_3)& \quad \text{if}~ x\in \Omega^+(0)\backslash I_0(\delta_0),\\
S\(\frac{d(x,I_0)}\ve\)(\mathrm{u}^{in}\otimes \mathrm{u}^{in}-\frac 13 I_3)& \quad \text{if}~ x\in I_0(\delta_0/2),\\
0&\quad \text{if}~ x\in \Omega^-(0)\backslash I_0(\delta_0).
\end{array}
\right.
\end{align}
Moreover,   there exists  a constant $c_1>0$ which only depends on $I_0$ and $\|\uu^{in}\|_{H^1(\O)}$  such that $Q_\ve^{in}$ is well-prepared in the sense of   \eqref{initial}.

  \end{prop}

The rest of this  work will be organized as follows. In Section \ref{sec pre} we discuss
some properties of the quasi-distance function \eqref{quasidistance} and use them to construct the well-prepared initial data \eqref{sharp initial} and thus prove Proposition \ref{prop initial data}. In Section \ref{sec entropy} we establish a relative-entropy type inequality for  the parabolic system \eqref{Ginzburg-Landau sys}. Based on the various estimates given by such an inequality, in Section \ref{sec har} we study   the  limit $\ve\downarrow 0$ of \eqref{Ginzburg-Landau sys} and give the proof of Theorem \ref{main thm}.

\section{Preliminaries}\label{sec pre}


We start with a lemma about  the quasi-distance function \eqref{quasidistance}, which was originally due to   \cite{MR930124,MR985992}.
\begin{lemma}\label{lem2}
The function $d^F(q)$
is locally Lipschitz in $\Q$ with point-wise derivative satisfying
  \begin{align}\label{eq:2.7}
|\nabla_q  d^F(q)|= \sqrt{2 F(q)}~~\text{for}~~a.e.~q\in\Q.\end{align}
Moreover,
\begin{align}
\label{eq:1.6}
   d^F(q)=\left\{
   \begin{array}{rl}
   0\qquad\text{if}&~q\in\N,\\
   c^F\qquad \text{if}&~q=0,
   \end{array}
   \right.
\end{align}
where $c^F$ is  the 1-d energy of the  minimal  connection  between $\N$ and $0$:
 \begin{equation}
 c^F:=\inf\left\{\int_0^1\(\frac{|\gamma'(t)|^2}2+ F(\gamma(t)\)\, dt \Big|   \gamma\in C^{0,1}([0,1];\Q),\gamma(0)\in \N,\gamma(1)=0\right\}.
\end{equation}
\end{lemma}

By elementary linear algebra, any $Q\in\Q$ can be expressed by $Q=\sum_{i=1}^3 \lambda_i(Q)P_i(Q)$ with $\sum_{i=1}^3 \lambda_i(Q)=0$,
   where $P_i(Q)=\mathrm{n}_i\otimes\nn_i$ denotes the projection onto the $i$-th eigenspace, and $\lambda_1(Q)\leq \lambda_2(Q)\leq \lambda_3(Q)$ are the eigenvalues ordered increasingly. Furthermore using the identities  $\sum_{j=1}^3 \lambda_j(Q)=0$ and $I_3=\sum_{j=1}^3 P_j(Q)$,  we may   write
\begin{align}\label{biaxial}
Q=\(s+\frac r3\)\left(P_3(Q)-\frac{1}{3} I_3\right)+\frac{2r}3\left(P_2(Q)-\frac{1}{3} I_3\right), \\
~\text{where}~\quad s(Q)=
\frac 32 \lambda_3(Q), ~r(Q) =\frac 32 \(\lambda_2(Q)-\lambda_1(Q)\).\label{biaxial1}
\end{align}
The next lemma gives a precise form of $d^F$ for uniaxial $Q$-tensors.
\begin{lemma}\label{lemma unixial}
  If   $Q=s_0\left(\mathrm{u} \otimes \mathrm{u}-\frac{1}{3} I_3\right)$ for some $s_0\in [0,s_+]$ and  $\mathrm{u}  \in \BS$, then
\begin{equation}\label{precise diff}
 d^F(Q)= \frac 2{\sqrt{3}} \int^{s_+}_{s_0} \sqrt{ f(\tau) } \, d\tau =:g(s_0),
\end{equation}
where $f(s)$ is given  by \eqref{uni bulk}.
 \end{lemma}
\begin{proof}
The argument here is similar to that in \cite{MR3624937}.
Let $\gamma$ be any curve connecting $\N$ to $Q$.    When expressed  in the form of eigenframe $\gamma(t)=\sum_{i=1}^3 \lambda_i(t)\nn_i(t)\otimes\nn_i(t),$
 we claim that $\mathrm{n}_i$ are constant in $t$. Actually using the identity
  \begin{equation}
  \lambda_{i}^{2}\left|\nn'_{i}\right|^{2}=\lambda_{i}^{2} \sum_{j=1}^{3}\left(\mathrm{n}_{j} \cdot  \nn'_{i}\right)^{2}=\lambda_{i}^{2} \sum_{j \neq i}\left(\mathrm{n}_{j} \cdot  \nn'_{i}\right)^{2},
  \end{equation}
  we   calculate
 \begin{align}
  | \gamma'|^{2}=&( \lambda_{1}')^{2}+(\lambda'_{2})^{2}+(\lambda'_{3})^{2}+2 \sum_{i=1}^3 \lambda_{i}^{2}\left|\nn_{i}'\right|^{2} +\sum_{k=1}^{3} \sum_{1 \leq i<j \leq 3} 4 \lambda_{i} \lambda_{j}\left(\mathrm{n}_{i} \cdot  \nn'_{j}\right)\left(\mathrm{n}_{j} \cdot  \nn'_{i}\right) \nonumber \\=&( \lambda_{1}')^{2}+(\lambda'_{2})^{2}+(\lambda'_{3})^{2}+\sum_{k=1}^{3} \sum_{1 \leq i<j \leq 3} 2\left(\lambda_{i}\left(\mathrm{n}_{j} \cdot \nn'_{i}\right)+\lambda_{j}\left(\mathrm{n}_{i} \cdot  \nn'_{j}\right)\right)^{2}\nonumber  \\ \geq&( \lambda_{1}')^{2}+(\lambda'_{2})^{2}+(\lambda'_{3})^{2}.\end{align}
  This implies that the  global minimum  is achieved by a path $\gamma(t)$ with  constant eigenframe.
      So by \eqref{biaxial} we may write
\begin{align}\label{diagonal}
\gamma(t)=\operatorname{diag}\left\{-\frac{ s(t)+r(t)}{3},-\frac{s(t)-r(t)}{3}, \frac{2s(t)}{3}\right\},\nonumber \\
\text{with}~(s,r)|_{t=0}=(s_+,0),\qquad (s,r)|_{t=1}=(s_0,0).
\end{align}
then by \eqref{bulkED}
\begin{align}\label{tilde F}
F(\gamma(t))=&\frac a9 (3s^2+r^2)+\frac c{81}(3s^2+r^2)^2-\frac{2b}{27}(s^3-sr^2)=:\tilde{F}(r,s).
\end{align}
Writing   $\sqrt{3} s+i r =:\rho e^{i\theta}$ with $i=\sqrt{-1}$,  we have $3\sqrt{3}(s^3-sr^2)=\rho^3\cos 3\theta$, and thus
 \begin{align*}
 & \int_0^1 |\gamma'(t)| \sqrt{2F(\gamma(t))}\, dt \nonumber\\
 =&\frac 23\int_0^1 \sqrt{3s'^2 + r'^2}\sqrt{\tilde{F}(s,r)}\, dt	\nonumber \\
 =&\frac 23\int_0^1 \sqrt{\rho'^2+\rho^2 \theta'^2}\sqrt{\frac {a\rho^2 }9+\frac {c \rho^4}{81}-\frac{2b \rho^3}{81\sqrt{3}}\cos 3\theta}\, dt.
   \end{align*}
It is clear that this energy is minimized when $\theta\equiv 0$, which corresponds to the uniaxial solution $r(t)\equiv 0$.  In view of \eqref{uni bulk}
\begin{equation}
  \int_0^1 |\gamma'(t)| \sqrt{2 F(\gamma(t))} \, dt\geq  \frac{2}{\sqrt{3}}\int_0^1  |s'(t)|\sqrt{f(s(t))}\, dt.
\end{equation}
 One can verify that the minimum of the right hand side  can be achieved by a monotone function $s(t)$, and thus   \eqref{precise diff} follows from  a change of variable.
\end{proof}
 At this point we would like to remark that for the general Keller--Rubinstein--Sternberg problem   it is very hard to obtain a precise form of $d^F$ like \eqref{precise diff} (cf. \cite[Part 2, Lemmas 5 and 7]{MR930124}).

 Before giving the proof of  Proposition \ref{prop initial data},  we digress here and
 discuss  the convolution in \eqref{psi}.
 The space $\Q$ \eqref{Q-space} can be equipped with the norm $|q|:=\sqrt{\tr (q^Tq)}$, and  one can easily verify that $\{E_i\}_{i=1}^5$ defined below form an orthonormal  basis:
\begin{align}
&E_{1}=\left[\begin{array}{ccc}
\frac{\sqrt{3}-3}{6} & 0 & 0 \\
0 & \frac{\sqrt{3}+3}{6} & 0 \\
0 & 0 & -\frac{\sqrt{3}}{3}
\end{array}\right], \quad E_{2}=\left[\begin{array}{ccc}
\frac{\sqrt{3}+3}{6} & 0 & 0 \\
0 & \frac{\sqrt{3}-3}{6} & 0 \\
0 & 0 & -\frac{\sqrt{3}}{3}
\end{array}\right],\nonumber\\
&E_{3}=\left[\begin{array}{ccc}
0 & \frac{\sqrt{2}}{2} & 0 \\
\frac{\sqrt{2}}{2} & 0 & 0 \\
0 & 0 & 0
\end{array}\right], \quad E_{4}=\left[\begin{array}{ccc}
0 & 0 & \frac{\sqrt{2}}{2} \\
0 & 0 & 0 \\
\frac{\sqrt{2}}{2} & 0 & 0
\end{array}\right], \quad E_{5}=\left[\begin{array}{ccc}
0 & 0 & 0 \\
0 & 0 & \frac{\sqrt{2}}{2} \\
0 & \frac{\sqrt{2}}{2} & 0
\end{array}\right].
\end{align}
So this establishes an isometry $\Q\simeq \R^5$ and thus the convolution operation in \eqref{psi} can be interpreted as an integration in  $\R^5$. Concerning the choice of $\phi$ in \eqref{psi convolu}, one can simply choose  $g\in C_c^\infty(\R)$ and set  $\phi(q):= g(\tr (q^2))$, which is obviously  isotropic in $q$.
\begin{proof}[Proof of Proposition \ref{prop initial data}]
As a consequence of the choice of the cutoff function $\zeta$ satisfying  \eqref{cut-off zeta}, we deduce that \eqref{transition initial data} is fulfilled and $\tilde{S}_\ve$ is smooth.  To compute the modulated energy  \eqref{entropy} of  the initial data $Q^{in}_\ve$, we  write \eqref{cut-off initial}  by
\begin{equation}
\tilde{S}_\ve(x)= S\(\frac{d(x,I_0)}\ve\)  +\hat{S}_\ve(x),
\end{equation}
where
\begin{equation}
\hat{S}_\ve(x):=\(1-\zeta\(\frac {d(x,I_0)}{\delta_0}\)\)\(s_+\1_{\Omega^+(0)} -S\(\frac{d(x,I_0)}\ve\)\).
\end{equation}
It follows from the exponential decay of \eqref{degree of orientation} that
\begin{equation}\label{hat S decay}
\|\hat{S}_\ve\|_{L^\infty(\O)}+\|\nabla \hat{S}_\ve\|_{L^\infty(\O)}\leq Ce^{-\frac C\ve},
\end{equation}
 for some constant $C>0$ that only depends on $I_0$.
So we can write
\begin{equation}
|\nabla Q^{in}_\ve|^2=\frac 23S'^2 +2S^2 |\nabla\uu^{in}|^2+O(e^{-C/\ve})(|\nabla\uu^{in}|^2+1)
\end{equation}
Recalling the form of the bulk energy \eqref{uni bulk}  for uniaxial $Q$-tensors, in view of \eqref{critical coeff}, we have
\begin{equation}\label{s-poly}
f(s)=\frac c 9s^2 (s-s_+)^2,\quad \sqrt{f(s)}=\frac {\sqrt{c}}3 s(s-s_+),~\text{for all}~ s\in [0,s_+],
\end{equation}
and $F(Q^{in}_\ve)=f(S +\hat{S}_\ve)$.
Thus the integrand of $E_\ve[Q_\ve | I](0)$ can  be written as 
\begin{align}
&\frac{\varepsilon}2 \left|\nabla Q_\ve^{in}\right|^2+\frac{1}{\varepsilon} F(Q_\ve^{in})-\frac{2S'}\ve\sqrt{\frac{f(S)}3}\nonumber \\
=&\frac {S'^2}{3\ve} +\frac{f(S)}\ve-\frac{2S'}\ve\sqrt{\frac{f(S)}3}\nonumber\\
&+\ve S^2 |\nabla\uu^{in}|^2+O(e^{-C/\ve})(|\nabla\uu^{in}|^2+1)+\frac{f(S +\hat{S}_\ve)-f(S)}\ve.
\end{align}
The first line on the right hand side vanishes due to the identity $S'(z)=\sqrt{3f(S(z))}$:
\begin{align}\label{initial data 1}
&\frac{\varepsilon}2 \left|\nabla Q_\ve^{in}\right|^2+\frac{1}{\varepsilon} F(Q_\ve^{in})-\frac{2S'}\ve\sqrt{\frac{f(S)}3}\nonumber \\
=&\ve S^2 |\nabla\uu^{in}|^2+O(e^{-C/\ve})(|\nabla\uu^{in}|^2+1)+\frac{f(S +\hat{S}_\ve)-f(S)}\ve.
\end{align} On the other hand, since $0\leq \tilde{S}_\ve\leq s_+$, by Lemma \ref{lemma unixial},
\begin{align}
d^F(Q^{in}_\ve(x))=\frac 2{\sqrt{3}} \int^{s_+}_{ \tilde{S}_\ve(x)} \sqrt{ f(\tau) } \, d\tau.
\end{align}
This together with   \eqref{cutoff func} and \eqref{hat S decay}  implies
\begin{align}\label{initial data 2}
&-(\xi\cdot\nabla)   d^F(Q_\ve^{in})\nonumber \\
=&-\eta (d(x,I_0))\frac{2S'}{\sqrt{3}\ve}\sqrt{ f(S)}-\xi\cdot \nn_I\frac{2S'}{\sqrt{3}\ve}\(\sqrt{ f(S)}-\sqrt{ f(S+\hat{S}_\ve)}\)+O(e^{-C/\ve}). \end{align}
Adding up \eqref{initial data 1} and \eqref{initial data 2} yields
\begin{align}\label{initial data 3}
&\frac{\varepsilon}2 \left|\nabla Q_\ve^{in}\right|^2+\frac{1}{\varepsilon} F(Q_\ve^{in})-(\xi\cdot\nabla)   d^F(Q_\ve^{in})\nonumber \\
=&\ve S^2 |\nabla\uu^{in}|^2+O(e^{-C/\ve})(|\nabla\uu^{in}|^2+1)+\frac{f(S +\hat{S}_\ve)-f(S)}\ve\nonumber\\
&+\(1-\eta (d(x,I_0))\)\frac{2S'}\ve\sqrt{\frac{f(S)}3}-\xi\cdot \nn_I\frac{2S'}\ve\(\sqrt{\frac{f(S)}3}-\sqrt{\frac{f(S+\hat{S}_\ve)}3}\).
\end{align}
By the exponential decay of \eqref{degree of orientation} and \eqref{hat S decay},
\begin{align}\label{initial data 4}
&\frac{\varepsilon \left|\nabla Q_\ve^{in}\right|}{2}^{2}+\frac{ F(Q_\ve^{in})}{\varepsilon}-(\xi\cdot\nabla)   d^F(Q_\ve^{in})\nonumber \\
\leq &\ve S^2 |\nabla\uu^{in}|^2+O(e^{-C/\ve})(|\nabla\uu^{in}|^2+1) +\(1-\eta (d(x,I_0))\)\frac{2S'}\ve\sqrt{\frac{f(S)}3}.\end{align}
To treat the last term, we first deduce from  the exponential decay of \eqref{degree of orientation}  that
\begin{equation}
\left\|\(\frac{d(x,I_0)}{\ve}\)^2S'\(\frac{d(x,I_0)}{\ve}\)\right\|_{L^\infty(I_0(\delta_0))}\leq C
\end{equation}
for some $C$ that only depends on $I_0$.
This together with   \eqref{cutoff func} implies
\begin{align}
&\(1-\eta (d(x,I_0))\)\frac{2S'}\ve\sqrt{\frac{f(S)}3}\nonumber\\
=& \ve^2 \frac{d^2(x,I_0)}{\ve^2}\frac{2S'}\ve\sqrt{\frac{f(S)}3}
-\eta (d(x,I_0))\1_{\{d(x,I_0)> \delta_0/2\}}\frac{2S'}\ve\sqrt{\frac{f(S)}3}\nonumber\\
\leq &\ve C(I_0)+Ce^{-\frac C\ve}.
\end{align}
Substituting  the above  estimate into \eqref{initial data 4} and use \eqref{new bulk},  we arrive at
\begin{align}\label{initial data 5}
\int_{\O}\(\frac{\varepsilon}2 \left|\nabla Q_\ve^{in}\right|^2+\frac{1}{\varepsilon}  {F_\ve(Q_\ve^{in})}-(\xi\cdot\nabla)   d^F(Q_\ve^{in})\)\, dx
\leq \(\ve+O(e^{-C/\ve})\) \int_{\O}|\nabla\uu^{in}|^2\, dx {+\ve^2|\Omega|}.\end{align}
On the other hand, by \eqref{psi convolu} and \eqref{eq:2.7}, we have
\begin{equation}
| d^F_\ve(q)-d^F(q)|=|(\phi_{\ve}* d^F)(q)-d^F(q)|\leq  \ve L,~\text{if}~|q|\leq M,
\end{equation}
where $L=L(M,\phi, F)$. This pointwise estimate implies
\begin{align}
\left|\int_{\O}\( (\xi\cdot\nabla)   d^F(Q_\ve^{in})-(\xi\cdot\nabla)   d^F_\ve(Q_\ve^{in})\)\, dx\right|
=\left|\int_{\O} (\div \xi) \(d^F(Q_\ve^{in})-  d^F_\ve(Q_\ve^{in})\)\, dx \right|\leq L\ve,
\end{align}
which  together with \eqref{initial data 5} implies \eqref{initial}.

\end{proof}

The next result is concerned with a maximum modulus estimate of \eqref{Ginzburg-Landau}
\begin{lemma}\label{L infinity bound}
  Assume $Q_\ve$ is the solution of \eqref{Ginzburg-Landau sys}    satisfying $\|Q_\ve^{in}\|_{L^\infty(\O)}\leq C_0$ for some fixed constant $C_0$. 
Then there exists an   $\ve$-independent constant $c_0=c_0(a,b,c,C_0)>0$  such that
\begin{equation}\label{L infinity bound1}
\|Q_\ve\|_{L^\infty(\Omega\times(0,T))}\leq c_0.
\end{equation}
\end{lemma}
\begin{proof}
On the one hand, by  \eqref{Ginzburg-Landau},   $|Q_\ve|^2$ fulfills   the following identity
\begin{equation}
\p_t |Q_\ve|^2-\Delta |Q_\ve|^2+|\nabla Q_\ve |^2=-\frac 2{\ve^2}
\(a|Q_\ve|^2-b\tr Q_\ve^3+c|Q_\ve|^4\).
\end{equation}
On the other hand, there exists $\mu>0$  (sufficiently large)  such that $|Q|\geq \mu$ implies \[ a|Q|^2-b\tr Q^3+c|Q|^4 > 0.\]
Assume $|Q_\ve|(x,t)$ achieves its maximum at $(x_\ve,t_\ve)\in \overline{\Omega\times (0,T)}$. If   $|Q_\ve(x_\ve,t_\ve)|\leq \mu$, then we obtain the desired estimate. Otherwise there holds $\p_t |Q_\ve|^2-\Delta |Q_\ve|^2\leq 0$, and  the weak maximum principle implies the maximum must be achieved on the parabolic boundary $\(\p\O\times (0,T)\) \cup \(\Omega\times \{0\}\)$, on which  $|Q_\ve|$ is bounded by our  assumptions.
\end{proof}


\section{The modulated energy  inequality}\label{sec entropy}

As the gradient flow of \eqref{GL energy}, the system
\eqref{Ginzburg-Landau} has the following energy dissipation law
\begin{equation}\label{dissipation}
E_\ve(Q_\ve(\cdot,T))+  \int_0^T \int_\O \ve |\p_t Q_\ve|^2 \,d x \,d t=E_\ve(Q_\ve^{in}(\cdot)),~\text{for all}~ T\geq 0.
\end{equation}
Due to the concentration  of $\nabla Q_\ve$ near  the interface $I_t$, this estimate is not sufficient to derive the convergence of $Q_\ve$. Following a recent work of  Fisher et al. \cite{fischer2020convergence} we shall develop in this section  a calibrated inequality, which modulates  the surface energy.

Recall in \eqref{def:xi}
 that
we extend the normal vector field $\operatorname{n}_{I}$ of the interface $I_t$ to a neighborhood of it.
   We  also extend the mean curvature vector $\mathrm{H}_I$ of \eqref{interface} to a neighborhood by
  	\begin{equation}\label{def:H}
  		\mathrm{H}_I(x,t) = \tilde \eta(d(x,I_t))\mathrm{H}_I(P_I(x,t),t)=\tilde \eta(d(x,I_t))(\div \nn_I)(P_I(x,t),t)\nn_I (P_I(x,t),t),
  	\end{equation}
  	where $\tilde \eta\in C^\infty_c((-\delta_I,\delta_I))$ is a cut-off which is identically equal to $1$ for $s\in(-\delta_I/2,\delta_I/2)$, and
	$P_I(x,t)=x-\nabla d(x,I_t) d(x,I_t)$ is the projection onto $I_t$.
 The definitions \eqref{def:xi} and \eqref{def:H} of $\xi$ and $\mathrm{H}_I$, respectively,  imply the following relations:
  \begin{subequations}\label{xi der}
  \begin{align}
\p_t \xi=&-\left(\mathrm{H}_{I} \cdot \nabla\right) \xi-\left(\nabla \mathrm{H}_{I}\right)^{T} \xi+O(d(x, I_t)),\label{xi der1} \\
\p_t |\xi|^{2}=&-\left(\mathrm{H}_{I} \cdot \nabla\right)|\xi|^{2}+O\left(d^{2}(x, I_t)\right),\label{xi der2}
\end{align}
  \end{subequations}
  where $\nabla \HH_I:=\{\p_j (\HH_I)_i\}_{1\leq i,j\leq d}$ is a matrix with $i$ being the row index.
Actually in $I_t(\delta_I/2)$ there holds  $\p_t d(x,I_t)=-\nn_I\cdot\HH_I(P_I(x,t))$ and $\nabla d(x,I_t)=\nn_I(P_I(x,t))$. So we obtain \eqref{xi der} by chain rule. Moreover,
\begin{align}
 -\div \xi=&\mathrm{H}_{I} \cdot \xi+O(d(x, I_t)),\label{div xi H}
 \end{align}
 and since $\HH_I$ is extended constantly in normal direction, we have
 \begin{align}
  (\xi\cdot\nabla )\HH_I&=0~\text{for all } (x,t) \text{ such that }|d(x,I_t)| <  \delta_I/2.\label{normal H}
  \end{align}
  {Moreover, by the choice of $\delta_I$ at the beginning of Section \ref{sec Main thm},  we have
  \begin{equation}\label{bc n and H}
  \xi=0~\text{on}~\p\O~\text{and}~\HH_I=0~\text{on}~\p\O.
  \end{equation}}
  Finally, we have the following regularity
\begin{align}
 |\nabla \xi| &+\left|\mathrm{H}_{I}\right|+\left|\nabla \mathrm{H}_{I}\right| \leq C(I_0).
  \end{align}
We denote the phase-field analogs of the mean curvature and  normal vectors by
\begin{subequations}
\begin{align}
 \mathrm{H}_\ve(x,t)&:=-\left(\varepsilon \Delta Q_\ve-\frac{\nabla_q F(Q_\ve)}{\varepsilon} \right):\frac{\nabla Q_\ve}{\left|\nabla Q_\ve\right|},
 \label{mean curvature app}\\
 \nn_\ve(x,t)&:=\frac{\nabla \psi_\ve(x,t)}{|\nabla \psi_\ve(x,t)|}, \label{normal diff}\end{align}
\end{subequations}
respectively,
where $\psi_\ve$ is defined by \eqref{psi}.
Here and throughout we use the convention that $:$ denotes the contraction in the indices $i,j$ in three-tensors like $\partial_k Q_{i,j}$, i.e., the scalar product in the state space $\Q$.

By  chain rule  and \eqref{psi}
\begin{align}
\label{ADM chain rule}
 \nabla \psi_\ve(x,t) & =  \nabla_q d^F_\ve(Q_\ve) \colon \nabla Q_\ve(x,t) \qquad \text{for a.e.\ }(x,t)\in \Omega\times (0,T).
\end{align}
This motivates the definition of the following projection of $\p_i Q_\ve$ onto  the span   of $\nabla_q d^F_\ve(Q_\ve)$
\begin{equation} \label{projection1}
\Pi_{Q_\ve}  \p_i Q_\ve=\left\{
\begin{array}{rl}
\(\p_i Q_\ve:\frac{\nabla_q d^F_\ve(Q_\ve)}{|\nabla_q d^F_\ve(Q_\ve)|}\) \frac{\nabla_q d^F_\ve(Q_\ve)}{|\nabla_q d^F_\ve(Q_\ve)|},&~\text{if}~\nabla_q d^F_\ve(Q_\ve)\neq 0,\\
0,&~\text{otherwise}.
\end{array}
\right.
\end{equation}
Hence, \eqref{ADM chain rule} implies
\begin{subequations}
\begin{align}
\label{projectionnorm}
|\nabla \psi_\ve| &= |\Pi_{Q_\ve} \nabla Q_\ve| |\nabla_q d^F_\ve(Q_\ve)| \qquad \qquad \text{for a.e.\ }(x,t)\in \Omega\times (0,T), \\
 \label{projection}
\Pi_{Q_\ve} \nabla Q_\ve&=\frac{|\nabla\psi_\ve|} {|\nabla_q d^F_\ve(Q_\ve)|^2}\nabla_q d^F_\ve(Q_\ve)\otimes \nn_\ve   \qquad \text{for a.e.\ }(x,t)\in \Omega\times (0,T),
\end{align}
\end{subequations}

The following inequality will be crucial to show the non-negativity of the modulated  energy \eqref{entropy} and various lower bounds of it. It states that the upper bound for the gradient of the convolution $d_\ve^F$ is as good as if $d^F$ was $C^{1,1/2}$ and it simply follows from the fact that  the modulus  $|\nabla d^F|$ is $C^{1/2}$.
\begin{lemma}\label{basic control est}
For each  $c_0>0$ there exists    $\ve_0\in \R^+$ such that    \begin{align}
	\label{sharp lip d}
	|\nabla_q d^F_\ve(q)| \leq \sqrt{2 F_\ve(q)},\qquad  \forall q\in\Q,\,  |q|\leq c_0,\, \forall \ve\in (0,\ve_0)\end{align}
\end{lemma}
\begin{proof}
Recall \eqref{new bulk}, i.e. $F_\ve(q)=F(q)+\ve^{K-1}$ with $K=4$.
It follows from   \eqref{psi convolu}, \eqref{eq:2.7}  and $\int_{\R^5} \phi_\ve(p)\, dp=1$ that 
	\begin{align*}
	|\nabla_q d^F_\ve(q)|&=\left|\int_{\R^5} \phi_\ve(p) \nabla_q d^F(q-p)\, dp\right|\\
	&\leq \int_{\R^5} \sqrt{\phi_\ve(p)}\sqrt{\phi_\ve(p)} \sqrt{2 F(q-p)}\, dp \nonumber\\
	&\leq \sqrt{\int_{\R^5} \phi_\ve(p) 2 F(q-p)\, dp}\\
	&\leq \sqrt{\int_{\R^5} \phi_\ve(p) \(2 F(q)+C_0|p|\)\, dp}	\end{align*}
	where in the last step $C_0$ is a local Lipschitz constant of $F(q)$ for $|q|\leq c_0$.
 By \eqref{psi convolu} and the assumption that $\phi$ is supported in the unit ball of $\Q$, the integral in the last step can be treated as follows 
 \begin{align*}
	|\nabla_q d^F_\ve(q)|	&\leq \sqrt{2F(q)+C_0 \ve^K\int_{\R^5} \phi_\ve(p)  |p|\ve^{-K}\, dp}\leq \sqrt{2F(q)+C_0 \ve^K}	\end{align*}
 Finally choosing   $\ve_0$ sufficiently small leads to  \eqref{sharp lip d}.\end{proof}
We shall apply the above lemma with $c_0$ being the constant in \eqref{L infinity bound1}.

	As we shall not integrate the time variable $t$ throughout this section,   we shall abbreviate the spatial integration $\int_\O$ by $\int$ and sometimes we omit the $\,dx$.
The following lemma shows that the energy $E_\ve [Q_\ve | I]$ defined by \eqref{entropy} controls various quantities.

\begin{lemma}\label{lemma:energy bound}
There exists a universal constant $C<\infty$ {which is independent of $t\in (0,T)$ and $\ve$  such that the following estimates hold  for  every $t\in (0,T)$:}
\begin{subequations} \label{energy bound}
\begin{align}
 \int \(\frac{\varepsilon}{2} \left|\nabla Q_\ve\right|^2+\frac{1}{\ve} F_\ve(Q_\ve)-|\nabla \psi_\ve| \)\, d x \leq & E_\ve [Q_\ve | I](t),\label{energy bound0}\\
  \frac12\int\left(\sqrt{\varepsilon}\left|\Pi_{Q_\ve}\nabla Q_\ve\right|-\frac{1}{\sqrt{\ve}}\sqrt{2 F_\ve(Q_\ve)} \right)^{2}\, d x+\frac \ve 2 \int \(  \left|\nabla Q_\ve-\Pi_{Q_\ve}\nabla Q_\ve\right|^2   \)\, d x\leq & E_\ve [ Q_\ve | I](t),\label{energy bound2}\\
  \frac12\int\left(\sqrt{\varepsilon}\left|\Pi_{Q_\ve}\nabla Q_\ve\right|-\frac1{\sqrt{\ve}} |\nabla_q d^F_\ve(Q_\ve)| \right)^{2}\, d x\leq & E_\ve [ Q_\ve | I](t),\label{energy bound2tilde}\\
   \int\left(\sqrt{\varepsilon}\left|\nabla Q_\ve\right|-\frac{1}{\sqrt{\varepsilon}} |\nabla_q d^F_\ve(Q_\ve)| \right)^{2} \, d x  \qquad \qquad\qquad\qquad \qquad\qquad&\nonumber\\
  \qquad \qquad\qquad\qquad \qquad\qquad
  +\int\left(1-\xi \cdot\nn_{\varepsilon}\right)\( {\frac{\ve}{2}}\left|\Pi_{Q_\ve}\nabla Q_\ve\right|^{2}+\left|\nabla \psi_{\varepsilon}\right|\) \, d x\leq   C&E_\ve [ Q_\ve | I](t),\label{energy bound1}
\\
   \int \(\frac{\varepsilon}2 \left|\nabla Q_\ve\right| ^{2}+\frac1 {\varepsilon}{ F_\ve(Q_\ve)}+|\nabla\psi_\ve|\) \min\(d^2(x,I_t),1\)\, d x\leq    C& E_\ve [ Q_\ve | I](t).
\label{energy bound3}
 \end{align}
\end{subequations}
\end{lemma}
 \begin{proof}[Proof of Lemma \ref{lemma:energy bound}]
	Since $|\xi\cdot\nabla \psi_\ve|\leq |\nabla\psi_\ve|$, we obtain the first estimate \eqref{energy bound0}.
	The second estimate \eqref{energy bound2} follows from the first one by using the chain rule in form of \eqref{projectionnorm} for the term $|\nabla \psi_\ve|$, the Lipschitz estimate \eqref{sharp lip d}  and then completing the square. Similarly, using the Lipschitz estimate \eqref{sharp lip d} to  the term $\frac1\ve F_\ve(Q_\ve)$ instead yields \eqref{energy bound2tilde}
	
	Let us now turn to the estimate \eqref{energy bound1}. Completing the square and using \eqref{sharp lip d} yield
	\begin{align} \label{entropy1}
	E_\ve [ Q_\ve | I]\geq  &\frac12\int \left(\sqrt{\varepsilon}\left|\nabla Q_\ve\right|-\frac{1}{\sqrt{\varepsilon}}|\nabla_q d^F_\ve(Q_\ve)| \right)^{2} d x+\int \(|\nabla_q d^F_\ve(Q_\ve)||\nabla Q_\ve|-\left|\nabla \psi_\ve\right| \)\, d x\nonumber \\
	&+\int\left(1-\xi \cdot\nn_{\varepsilon}\right)\left|\nabla \psi_{\varepsilon}\right| \, d x.
	\end{align}
	By the chain rule  in form of \eqref{projectionnorm}, the second right-hand side integral is non-negative.
	Using  \eqref{projection} and Young's inequality, it holds
	\begin{align}
	\varepsilon\left|\Pi_{Q_\ve}\nabla Q_\ve\right|^{2}&=\left|\nabla \psi_{\varepsilon}\right|+\sqrt{\varepsilon}\left|\Pi_{Q_\ve}\nabla Q_\ve\right|\left(\sqrt{\varepsilon}\left|\Pi_{Q_\ve}\nabla Q_\ve\right|-\frac{ |\nabla_q d^F_\ve(Q_\ve)|}{\sqrt{\varepsilon}}\right)\nonumber \\
	&\leq\left|\nabla \psi_{\varepsilon}\right|+\frac{\ve}{2}  \left|\Pi_{Q_\ve} \nabla Q_\ve\right|^{2}+\frac{1}{2}\left(\sqrt{\varepsilon}\left|\Pi_{Q_\ve}\nabla Q_\ve\right|-\frac{ |\nabla_q d^F_\ve(Q_\ve)|}{\sqrt{\varepsilon}} \right)^{2}.
	\end{align}
	Hence
	\begin{align}
	\frac\varepsilon2 \left|\Pi_{Q_\ve}\nabla Q_\ve\right|^{2}&\leq \left|\nabla \psi_{\varepsilon}\right|+\frac12 \left(\sqrt{\varepsilon}\left|\Pi_{Q_\ve}\nabla Q_\ve\right|-\frac{ |\nabla_q d^F_\ve(Q_\ve)|}{\sqrt{\varepsilon}} \right)^2.\label{young}
	\end{align}
	This combined with \eqref{entropy1}, \eqref{energy bound2tilde} and the trivial estimate $1-\xi\cdot\nn_{\varepsilon}\leq 2$ leads to \eqref{energy bound1}.
	Finally, by \eqref{cutoff func} we have \[1-\xi \cdot\nn_\ve\geq 1-\eta  \geq \min(d^2(x,I_t),1).\]
Applying this     to the second right-hand side integral of \eqref{energy bound1} and then using \eqref{energy bound0}  yield
	\eqref{energy bound3}.
\end{proof}

The following result was first proved in \cite{fischer2020convergence} in the case of the Allen-Cahn equation, and can be generalized to the vectorial case.
\begin{prop}\label{gronwallprop}
	There exists a constant $C=C(I_t)$ depending on the interface $I_t$ such that
	\begin{align}
	\frac{d}{d t} E_\ve [ Q_\ve | I] &+\frac 1{2\ve}\int \(\ve^2 \l| \p_t Q_\ve  \r|^2-|\HH_\ve|^2\)\,dx+\frac 1{2\ve}\int \l| \ve\p_t Q_\ve  -(\nabla\cdot  \xi )\nabla_q d^F_\ve(Q_\ve)    \r|^2\,dx\nonumber \\
	&+\frac 1{2\ve}\int \Big| \HH_\ve-\ve  |\nabla Q_\ve|\HH_I \Big|^2\,dx  \leq CE_\ve [ Q_\ve | I]. \label{gronwall}
	\end{align}
\end{prop}

The following lemma, the proof of which will be given at the end of this section,  provides the exact computation of the time derivative of the energy $E_\ve[Q_\ve|I ]$.
\begin{lemma}\label{lemma exact dt relative entropy}
	Under the assumptions of Theorem \ref{main thm}, the following identity holds
	\begin{subequations}\label{time deri 4}
		\begin{align}
		\frac{d}{d t} E&\left[Q_\ve | I\right]
		+\frac 1{2\ve}\int \l| \ve \p_t Q_\ve  -(\nabla \cdot \xi) \nabla_q d^F_\ve(Q_\ve)  \r|^2d x
		+\frac 1{2\ve}\int \big| \HH_\ve-\ve|\nabla Q_\ve| \HH_I \big|^2\,d x\nonumber \\
		=&\frac 1{2\ve} \int \Big| (\nabla \cdot \xi) |\nabla_q d^F_\ve(Q_\ve)|\nn_\ve +\ve |\Pi_{Q_\ve} \nabla Q_\ve| \HH_I\Big|^2\,d x\label{tail1}
		\\&+\frac \ve{2} \int |\HH_I|^2\(|\nabla Q_\ve|^2-|\Pi_{Q_\ve}\nabla Q_\ve|^2\)\,d x
		-\int \nabla \mathrm{H}_{I}: (\xi-\nn_\ve)^{\otimes 2}\left|\nabla \psi_{\varepsilon}\right|\,d x\label{tail2}\\
		&   +\int \(\nabla\cdot\HH_I\)  \( \frac{\ve}2 |\nabla Q_\ve|^2 +\frac{1}\ve F_\ve(Q_\ve) -|\nabla \psi_\ve| \)\,d x+\int\(\nabla\cdot\HH_I\)  \(1-\xi\cdot \nn_\ve\)|\nabla\psi_\ve|\, d x+ J_\ve^1+ J_\ve^2,\label{tail3}
		\end{align}
	\end{subequations}
	where we use the notation
	\begin{align}
	J_\ve^1
	:=&\int \nabla \mathrm{H}_{I}: \nn_\ve \otimes\nn_{\varepsilon}\(|\nabla \psi_\ve|-\ve |\nabla Q_\ve|^2\)\, dx+\ve \int \nabla \HH_I:(\nn_\ve\otimes \nn_\ve)\(
	|\nabla Q_\ve|^2-|\Pi_{Q_\ve} \nabla Q_\ve|^2\)\, dx \nonumber\\
	&-\ve\int \sum_{i,j=1}^3(\nabla \HH_I)_{ij}   \Big((\p_i Q_\ve-\Pi_{Q_\ve} \p_i Q_\ve):(\p_j Q_\ve-\Pi_{Q_\ve} \p_j Q_\ve)\Big)\, dx ,\label{J1}\\
	J_\ve^2:= &-\int \(\p_t  \xi+\left(\mathrm{H}_{I} \cdot \nabla\right) \xi+\left(\nabla \mathrm{H}_{I}\right)^{T} \xi\)\cdot (\nn_\ve-\xi) |\nabla \psi_\ve|\, dx\nonumber \\
	&-\int \Big(\p_t  \xi+\left(\mathrm{H}_{I} \cdot \nabla\right) \xi\Big)\cdot \xi\, |\nabla \psi_\ve|\, dx.\label{J2}
	\end{align}
\end{lemma}

In order to prove the proposition, we only need to estimate the terms on the RHS of \eqref{time deri 4}.

\begin{proof}[Proof of Proposition \ref{gronwallprop}]
	According to Lemma \ref{lemma exact dt relative entropy}, we only need to estimate the RHS of \eqref{time deri 4} by $E_\ve [Q_\ve | I]$ up to a constant that only depends on $I_t$. We start with \eqref{tail1}: it follows from the  triangle inequality that
	\begin{equation*}
	\begin{split}
	\frac 1{2} \int& \left|\frac1{\sqrt{\ve}} (\nabla \cdot \xi) |\nabla_q d^F_\ve(Q_\ve)|\nn_\ve +\sqrt{\ve} |\Pi_{Q_\ve} \nabla Q_\ve| \HH_I\right|^2d x
\\	\leq &\int \left|(\nabla\cdot \xi)
	\left(
	 \sqrt{\ve} |\Pi_{Q_\ve} \nabla Q_\ve| - \frac{1}{\sqrt{\ve}} |\nabla_q d^F_\ve(Q_\ve)|
	\right)
	\nn_\ve\right|^2
	d x
	\\&+\int \left|(\nabla\cdot \xi)
	\sqrt{\ve}  |\Pi_{Q_\ve} \nabla Q_\ve|
	(\nn_\ve-\xi)\right|^2
	d x
	\\&+ \int \left|
	(\HH_I +(\nabla \cdot \xi) \xi )  \sqrt{\ve} |\Pi_{Q_\ve} \nabla Q_\ve| \right|^2
	d x.
	\end{split}
	\end{equation*}
	The first  integral is controlled by \eqref{energy bound2tilde}. Using \eqref{div xi H},
	the second integral is controlled by \eqref{energy bound1}. The third integral can be treated using \eqref{div xi H} and controlled by \eqref{energy bound3}.

		The integrals in \eqref{tail2} can be controlled using \eqref{energy bound2} and \eqref{energy bound1}, recalling
	\[|\xi - \nn_\ve|^2 \leq 2 (1-\nn_\ve\cdot\xi).\]
	The first term in \eqref{tail3} can be controlled using \eqref{energy bound0}, and the second term can be  estimated by \eqref{energy bound1}.  It remains to estimate   \eqref{J1} and \eqref{J2}. The last two terms on the RHS of $J_\ve^1$ can be bounded using  \eqref{energy bound2}, and the first integral can be rewritten using $\nn_\ve=\nn_\ve-\xi+\xi$:
	\begin{align}
	J_\ve^1
	\leq &\int \nabla \mathrm{H}_{I}: \( \nn_\ve \otimes (\mathrm{n}_\ve-\xi)\) \(|\nabla \psi_\ve|-\ve |\nabla Q_\ve|^2\)\, dx\nonumber\\
	&+\int \nabla \mathrm{H}_{I}: \nn_\ve \otimes \xi \(|\nabla \psi_\ve|-\ve |\nabla Q_\ve|^2\)\, dx+ C E_\ve [Q_\ve | I]\nonumber\\
	\leq &\|\nabla\HH_I\|_{L^\infty}\int  |\mathrm{n}_\ve-\xi| \(\ve |\nabla Q_\ve|^2-\ve |\Pi_{Q_\ve}\nabla Q_\ve|^2+\left|\ve |\Pi_{Q_\ve}\nabla Q_\ve|^2-|\nabla \psi_\ve|\right|\)\, dx\nonumber\\
	&+C\int\min \(d^2(x,I_t),1\)  \(|\nabla \psi_\ve|+\ve |\nabla Q_\ve|^2\)\, dx+ C E_\ve [Q_\ve | I].
	\end{align}
	Note that in the last step we employed
	\[\nabla \mathrm{H}_{I}: \nn_\ve \otimes \xi=(\xi\cdot \nabla \HH_I)\cdot \nn_\ve\]
	and the fact that $(\xi\cdot \nabla) \HH_I$ vanishes in the neighborhood $I_t(\frac{\delta_I}{2})$  by definition \eqref{def:H}.
	So we employ \eqref{energy bound1} and \eqref{energy bound2}, and \eqref{projectionnorm}
	\begin{align}
	J_\ve^1
	\leq &C\int  |\mathrm{n}_\ve-\xi| \left| \ve |\Pi_{Q_\ve}\nabla Q_\ve|^2-|\nabla \psi_\ve|\right|\, dx+ C E_\ve [Q_\ve | I]\nonumber\\
	= &C\int  |\mathrm{n}_\ve-\xi| \sqrt{\ve} |\Pi_{Q_\ve}\nabla Q_\ve| \left| \sqrt{\ve} |\Pi_{Q_\ve}\nabla Q_\ve|-\frac{|\nabla_q d^F_\ve(Q_\ve)|}{\sqrt{\ve}}\right|\, dx+ C E_\ve [Q_\ve | I].
	\end{align}
	Finally applying the Cauchy-Schwarz inequality and then \eqref{energy bound1} and \eqref{energy bound2}, we obtain
	\begin{equation}
	J_\ve^1 \leq C  E_\ve [Q_\ve | I].
	\end{equation}
	As for $J_\ve^2$ \eqref{J2}, we employ \eqref{xi der1} and \eqref{xi der2} and yield
	\begin{equation}
	J_\ve^2 \leq C  \int \(  |\nn_\ve-\xi|^2+d^2(x,I_t)\)|\nabla\psi_\ve|\, dx\leq CE_\ve [Q_\ve| I],
	\end{equation}
	after applying \eqref{energy bound1} and \eqref{energy bound3}.  So we proved that the RHS of \eqref{time deri 4} is bounded by $E_\ve [Q_\ve | I]$ up to a multiplicative constant which only depends on $I_t$.
\end{proof}


The following lemma will be used in the proof of Lemma \ref{lemma exact dt relative entropy}.
\begin{lemma}\label{lemma:expansion 1} Under the assumptions of Theorem \ref{main thm},
	\begin{align}
	&\int \nabla \mathrm{H}_{I}: (\xi \otimes\nn_{\varepsilon})\left|\nabla \psi_{\varepsilon}\right| \, d x
	-\int (\nabla \cdot \mathrm{H}_{I}) \, \xi  \cdot \nabla \psi_{\varepsilon} \, d x\nonumber\\
	=&\int \nabla \mathrm{H}_{I}: (\xi-\nn_\ve) \otimes\nn_{\varepsilon}\left|\nabla \psi_{\varepsilon}\right|\, d x+\int \HH_\ve\cdot \HH_I |\nabla Q_\ve|\, d x \nonumber\\
	&+\int\nabla\cdot\HH_I  \( \frac{\ve}2 |\nabla Q_\ve|^2  +\frac{1}\ve F_\ve(Q_\ve) -|\nabla \psi_\ve| \)\, d x +\int\nabla\cdot\HH_I ( |\nabla\psi|-\xi\cdot\nabla\psi_\ve)\, d x\nonumber\\
	&-\int \sum_{i,j=1}^3(\nabla \HH_I)_{ij}\, \ve\(\p_i Q_\ve: \p_j Q_\ve   \)\, d x +\int \nabla \mathrm{H}_{I}: (\nn_\ve \otimes\nn_{\varepsilon})\left|\nabla \psi_{\varepsilon}\right|\, d x\label{expansion1}
	\end{align}
\end{lemma}
\begin{proof}
	The LHS of \eqref{expansion1} can be written as
	\begin{align}
	&\int \nabla \mathrm{H}_{I}: \(\xi \otimes\nn_{\varepsilon}\)\left|\nabla \psi_{\varepsilon}\right| \, d x
	-\int (\nabla \cdot \mathrm{H}_{I}) \, \xi  \cdot \nabla \psi_{\varepsilon} \, d x\nonumber\\
	=&\int \nabla \mathrm{H}_{I}: (\xi-\nn_\ve) \otimes\nn_{\varepsilon}\left|\nabla \psi_{\varepsilon}\right|\, d x+\int \nabla \mathrm{H}_{I}: \(\nn_\ve \otimes\nn_{\varepsilon}\)\left|\nabla \psi_{\varepsilon}\right|\, d x-\int (\nabla \cdot \mathrm{H}_{I})\, \xi  \cdot \nabla \psi_{\varepsilon} \, d x.\label{merge}
	\end{align}
	To treat the second term on the RHS of \eqref{expansion1}, we
	introduce the energy stress tensor $T_\ve$
	\begin{equation}
	\begin{split}
	(T_\ve)_{ij}=& \( \frac{\ve}2 |\nabla Q_\ve|^2 +\frac{1}{\ve} F_\ve(Q_\ve)  \) \delta_{ij} - \ve \p_i Q_\ve: \p_j Q_\ve.   \end{split}
	\end{equation}
	In view of \eqref{mean curvature app},
	we have the identity
	\begin{equation}\label{variation1}
	\nabla \cdot T_\ve
	=-\ve \nabla Q_\ve : \Delta Q_\ve
	+ \frac{1}{\ve} \nabla_q F_\ve(Q_\ve) : \nabla Q_\ve=\HH_\ve |\nabla Q_\ve|.
	\end{equation}
	Testing this identity by $\HH_I$, integrating by parts and using \eqref{bc n and H}, we obtain
	\begin{equation}
	\begin{split}
	\int \HH_\ve\cdot \HH_I |\nabla Q_\ve|\,d x&  =- \int \nabla \HH_I \colon T_\ve \,d x,\\
	&=-  \int\nabla\cdot\HH_I  \( \frac{\ve}2 |\nabla Q_\ve|^2 +\frac{1}{\ve} F_\ve(Q_\ve)  \)\, dx+ \int \sum_{i,j=1}^3(\nabla \HH_I)_{ij} \, \ve\(\p_i Q_\ve: \p_j Q_\ve\) d x.
	\end{split}
	\end{equation}
	So adding zero leads to
	\begin{equation}
	\begin{split}
	&\int \nabla \mathrm{H}_{I}: \nn_\ve \otimes\nn_{\varepsilon}\left|\nabla \psi_{\varepsilon}\right|d x\\
	&=\int \HH_\ve\cdot \HH_I |\nabla Q_\ve|\, dx+\int\nabla\cdot\HH_I  \( \frac{\ve}2 |\nabla Q_\ve|^2 +\frac{1}{\ve} F_\ve(Q_\ve) -|\nabla \psi_\ve| \)\,d x+\int\nabla\cdot\HH_I |\nabla\psi_\ve|\,d x\\
	&-\int \sum_{i,j=1}^3(\nabla \HH_I)_{ij}\,\ve\(\p_i Q_\ve: \p_j Q_\ve   \) d x+\int (\nabla \mathrm{H}_{I}): (\nn_\ve \otimes\nn_{\varepsilon})\left|\nabla \psi_{\varepsilon}\right|d x.
	\end{split}
	\end{equation}
	Substituting this identity  into \eqref{merge} leads to \eqref{expansion1}.
\end{proof}

\begin{proof}[Proof of Lemma \ref{lemma exact dt relative entropy}]
	Using the energy dissipation law  \eqref{dissipation} and adding zero, we compute the time derivative of the energy \eqref{entropy} by
	\begin{align*}
	&\frac{d}{d t} E_\ve [ Q_\ve | I]
	+\ve\int |\p_t Q_\ve|^2\,d x-\int (\nabla \cdot \xi)   \nabla_q d^F_\ve(Q_\ve): \p_t Q_\ve \,d x\nonumber\\
	=&\int   \left(\mathrm{H}_{I} \cdot \nabla\right) \xi\cdot\nabla \psi_\ve\,d x
	+\int \left(\nabla \mathrm{H}_{I}\right)^{T} \xi \cdot\nabla \psi_\ve\,d x
	-\int \(\p_t  \xi+\left(\mathrm{H}_{I} \cdot \nabla\right) \xi 
	+\left(\nabla \mathrm{H}_{I}\right)^{T} \xi\)\cdot\nabla \psi_\ve\,d x.\label{time deri 1}
	\end{align*}
	Due to the symmetry of the Hessian of $\psi_\ve$  and the boundary conditions \eqref{bc n and H}, we have
	\begin{align*}
	\int \nabla \cdot (\xi \otimes \mathrm{H}_I ) \cdot \nabla \psi_\ve \, d x =  \int \nabla \cdot (\mathrm{H}_I \otimes \xi) \cdot \nabla \psi_\ve \, d x.
	\end{align*}
	Hence, the first integral on the RHS above can be rewritten as
	\begin{align*}
	\int\left(\mathrm{H}_{I} \cdot \nabla\right) \,\xi \cdot \nabla \psi_{\varepsilon} \, d x
	&=\int \nabla \cdot (\xi \otimes \mathrm{H}_I ) \cdot \nabla \psi_\ve \, d x
	-\int (\nabla \cdot \mathrm{H}_{I}) \,\xi \cdot \nabla \psi_{\varepsilon} \, d x\nonumber \\
	&= \int(\nabla \cdot \xi) \,\mathrm{H}_{I} \cdot \nabla \psi_{\varepsilon} \, d x
	+\int(\xi \cdot \nabla) \,\mathrm{H}_{I} \cdot \nabla \psi_{\varepsilon} \, d x
	-\int (\nabla \cdot \mathrm{H}_{I}) \,\xi  \cdot \nabla \psi_{\varepsilon} \,d x.
	\end{align*}
	Therefore
	\begin{equation*}
	\begin{split}
	&\frac{d}{d t} E_\ve [ Q_\ve | I]
	+\ve\int |\p_t Q_\ve|^2\,d x-\int (\nabla \cdot \xi)   \nabla_q d^F_\ve(Q_\ve): \p_t Q_\ve\,d x \\
	=& \int (\nabla \cdot \xi)\, \mathrm{H}_{I} \cdot \nabla \psi_\ve d x+\int(\xi \cdot \nabla) \,\mathrm{H}_{I} \cdot \nabla \psi_{\varepsilon} \,d x
	-\int (\nabla \cdot \mathrm{H}_{I})\, \xi  \cdot \nabla \psi_{\varepsilon} \,d x\\
	&   +\int \nabla \mathrm{H}_{I}: \(\xi \otimes\nn_{\varepsilon}\)\left|\nabla \psi_{\varepsilon}\right| d x-\int \(\p_t  \xi+\left(\mathrm{H}_{I} \cdot \nabla\right) \xi 
	+\left(\nabla \mathrm{H}_{I}\right)^{T} \xi\)\cdot\nabla \psi_\ve \,d x
	\end{split}
	\end{equation*}
	Using   \eqref{expansion1} in Lemma \ref{lemma:expansion 1} to replace the third and fourth integrals on the RHS of the above identity and rewriting the last  integral, we arrive at
	\begin{equation}\label{time deri 3-}
	\begin{split}
	&\frac{d}{d t} E_\ve [ Q_\ve | I]
	+\ve\int |\p_t Q_\ve|^2\,d x-\int  (\nabla \cdot \xi)  \nabla_q d^F_\ve(Q_\ve): \p_t Q_\ve\,d x
	\\   =& \int(\nabla \cdot \xi) \,\mathrm{H}_{I}  \cdot \nabla \psi_{\varepsilon}\,d x
	+\int(\xi \cdot \nabla)\, \mathrm{H}_{I} \cdot \nabla \psi_{\varepsilon}\,d x
	+\int \nabla \mathrm{H}_{I}: (\xi-\nn_\ve) \otimes\nn_{\varepsilon}\left|\nabla \psi_{\varepsilon}\right|d x\\
	&   +\int \HH_\ve\cdot \HH_I |\nabla Q_\ve|\,d x
	+ \int\nabla\cdot\HH_I  \( \frac{\ve}2 |\nabla Q_\ve|^2 +\frac{1}\ve F_\ve(Q_\ve) -|\nabla \psi_\ve| \)\,d x
	+\int\nabla\cdot\HH_I \(|\nabla\psi_\ve|-\xi\cdot \nabla\psi_\ve\)d x\\
	&-\int \sum_{i,j=1}^3(\nabla \HH_I)_{ij}\, \ve \(\p_i Q_\ve: \p_j Q_\ve  \) d x
	+\int \nabla \mathrm{H}_{I}: \nn_\ve \otimes\nn_{\varepsilon}\left|\nabla \psi_{\varepsilon}\right|d x\\
	&-\int \(\p_t  \xi+\left(\mathrm{H}_{I} \cdot \nabla\right) \xi+\left(\nabla \mathrm{H}_{I}\right)^{T} \xi\)\cdot (\nn_\ve-\xi) |\nabla \psi_\ve|\,d x\\
	&-\int \Big(\p_t  \xi+\left(\mathrm{H}_{I} \cdot \nabla\right) \xi\Big)\cdot \xi \,|\nabla \psi_\ve|\, dx -\int \left(\nabla \mathrm{H}_{I}\right)^{T} :(\xi\otimes \xi) |\nabla \psi_\ve|\,d x.
	\end{split}
	\end{equation}
	First, note that the third to last and second to last integrals combine to $J_\ve^2$.
	Next, by the property \eqref{projection} of the orthogonal projection \eqref{projection1}, we can also find $J_\ve^1$ on the right-hand side. Indeed,
	\begin{equation*}
	\begin{split}
	-&\int \sum_{i,j=1}^3 (\nabla \HH_I)_{ij} \,\ve \(\p_i Q_\ve: \p_j Q_\ve  \)\,dx +\int \nabla \mathrm{H}_{I}: \nn_\ve \otimes\nn_{\varepsilon}\left|\nabla \psi_{\varepsilon}\right|\,dx\\
	&=\int \nabla \mathrm{H}_{I}: \nn_\ve \otimes\nn_{\varepsilon}\left|\nabla \psi_{\varepsilon}\right|\,dx-\ve\int (\nabla \HH_I)_{ij}(\Pi_{Q_\ve} \p_i Q_\ve :\Pi_{Q_\ve} \p_j Q_\ve)\,dx\\
	&\quad-\int \sum_{i,j=1}^3(\nabla \HH_I)_{ij} \,\ve \Big((\p_i Q_\ve-\Pi_{Q_\ve} \p_i Q_\ve):(\p_j Q_\ve-\Pi_{Q_\ve} \p_j Q_\ve)\Big) \, dx\\
	&=\int \nabla \mathrm{H}_{I}: \nn_\ve \otimes\nn_{\varepsilon}\(|\nabla \psi_\ve|-\ve |\nabla Q_\ve|^2\)\,dx+\ve \int \nabla \HH_I:(\nn_\ve\otimes \nn_\ve)\(
	|\nabla Q_\ve|^2-|\Pi_{Q_\ve} \nabla Q_\ve|^2\)\, dx \\
	&\quad- \int \sum_{i,j=1}^3(\nabla \HH_I)_{ij} \,\ve \Big((\p_i Q_\ve-\Pi_{Q_\ve} \p_i Q_\ve):(\p_j Q_\ve-\Pi_{Q_\ve} \p_j Q_\ve)\Big) \, dx =J_\ve^1.
	\end{split}
	\end{equation*}
	Using the definition \eqref{normal diff} of $\mathrm{n}_\ve$, we may merge  the second, third, and the last integral on the RHS of \eqref{time deri 3-} to obtain
	\begin{equation}\label{time deri 3}
	\begin{split}
	\frac{d}{d t} E_\ve [ Q_\ve | I]
	=&-\ve\int |\p_t Q_\ve|^2\, dx+\int  (\nabla \cdot \xi)  \nabla_q d^F_\ve(Q_\ve): \p_t Q_\ve\, dx\\
	&+ \int(\nabla\cdot \xi) \,\mathrm{H}_{I}  \cdot \nabla \psi_{\varepsilon}\, dx+\int \HH_\ve\cdot \HH_I |\nabla Q_\ve| \, dx -\int \nabla \mathrm{H}_{I}: (\xi-\nn_\ve)^{\otimes 2}\left|\nabla \psi_{\varepsilon}\right|\, dx\\
	&+ J_\ve^1   +\int \(\nabla\cdot\HH_I\)  \Big( \frac{\ve}2 |\nabla Q_\ve|^2 +\frac{1}\ve F_\ve(Q_\ve) -|\nabla \psi_\ve| \Big)\, dx\nonumber\\&+\int(\nabla\cdot\HH_I) \(1-\xi\cdot \nn_\ve\)|\nabla\psi_\ve|\, dx+J_\ve^2.
	\end{split}
	\end{equation}
	
	
	Now we complete squares for the first four terms on the RHS of \eqref{time deri 3}: Reordering terms, we have
	\begin{align}
	\notag-&\ve |\p_t Q_\ve|^2+   (\nabla \cdot \xi)  \nabla_q d^F_\ve(Q_\ve): \p_t Q_\ve
	+ (\nabla \cdot \xi) \mathrm{H}_{I}  \cdot \nabla \psi_{\varepsilon}+ \HH_\ve\cdot \HH_I |\nabla Q_\ve|
	\\\notag&= -\frac1{2\ve} \Big(  |\ve \p_t Q_\ve|^2 -2(\nabla \cdot \xi)  \nabla_q d^F_\ve(Q_\ve): \ve \p_t Q_\ve
	+(\nabla \cdot \xi)^2 | \nabla_q d^F_\ve(Q_\ve)|^2 \Big)
	\\\notag&\quad - \frac1{2\ve} |\ve \p_t Q_\ve|^2 + \frac1{2\ve}(\nabla \cdot \xi)^2 | \nabla_q d^F_\ve(Q_\ve)|^2
	+ (\nabla \cdot \xi) \mathrm{H}_{I}  \cdot \nabla \psi_{\varepsilon}
	\\\notag&\quad - \frac1{2\ve} \Big( |\HH_\ve|^2 - 2\HH_\ve\cdot \ve |\nabla Q_\ve| \HH_I + \ve^2 |\nabla Q_\ve|^2 |\HH_I|^2\Big)
	+ \frac1{2\ve} \Big( |\HH_\ve|^2 + \ve^2 |\nabla Q_\ve|^2 |\HH_I|^2\Big)
	\\&\notag =  -\frac1{2\ve} \Big|\ve \p_t Q_\ve- (\nabla \cdot \xi) \nabla_q d^F_\ve(Q_\ve) \Big|^2
	- \frac1{2\ve} \Big|\HH_\ve - \ve |\nabla Q_\ve| \HH_I \Big|^2
	- \frac1{2\ve}  |\ve \p_t Q_\ve|^2 +\frac1{2\ve}  |\HH_\ve|^2
	\\\label{square1}&\quad + \frac1{2\ve} \Big( (\nabla \cdot \xi)^2  |\nabla_q d^F_\ve(Q_\ve)|^2 + 2\ve(\nabla \cdot \xi)  \nabla \psi_\ve \cdot \HH_I + |\ve \Pi_{Q_\ve}\nabla Q_\ve|^2 |\HH_I|^2 \Big)
	\\\notag&\quad +\frac\ve{2} \left(|\nabla Q_\ve|^2- | \Pi_{Q_\ve}\nabla Q_\ve|^2\right)  |\HH_I|^2.
	\end{align}
	Using the definition \eqref{normal diff} of the normal $\nn_\ve$ and the chain rule in form of \eqref{projectionnorm}, the terms in  \eqref{square1} form the last missing square.
	Integrating over the domain $\Omega$ and substituting into \eqref{time deri 3} we arrive at \eqref{time deri 4}.
\end{proof}

\section{Convergence to the harmonic map  heat flow}\label{sec har}
This section is devoted to the proof of Theorem \ref{main thm}.
We start with a lemma about   uniform estimates of $Q_\ve$.
\begin{lemma}
There exists a universal constant $C=C(I_0)$ such that   
 \begin{equation}\label{energy bound4}
{\operatorname{ess~sup}}_{t\in  [0,T]} \int_\O \(  \l|\nabla Q_\ve-\Pi_{Q_\ve}\nabla Q_\ve\r|^2   \)\, dx+ \int_0^T\int_\O \(  \l|\p_t Q_\ve-\Pi_{Q_\ve}\p_t Q_\ve\r|^2   \)\, dxdt\leq e^{(1+T)C(I_0)}.
 \end{equation}
 Moreover,  for any fixed $\delta\in (0, \delta_I)$, there holds
 \begin{equation}\label{space der bound local}
{\rm{ess~sup}}_{t\in  [0,T]}\int_{\Omega^\pm(t)\backslash I_t(\delta)}\(|\nabla Q_\ve|^2+\frac{{F(Q_\ve)+\ve^3}}{\ve^2}\)\, dx \leq \delta^{-2}e^{(1+T)C(I_0)},
\end{equation}

\begin{equation}\label{time der bound local}
\int_0^T\int_{\Omega^\pm(t)\backslash I_t(\delta)} |\p_t Q_\ve|^2\, dx dt\leq \delta^{-2}e^{(1+T)C(I_0)}.
\end{equation}
\end{lemma}

\begin{proof}
We first establish a priori estimates of the solutions $Q_\ve$ which are independent of $\ve$. 
It follows from \eqref{gronwall} and the assumption \eqref{initial}  that
\begin{align}\label{energy1}
& {\rm{ess~sup}}_{t\in  [0,T]} \frac{1}\ve E_\ve [ Q_\ve | I] (t)+\frac 1{\ve^2}\int_0^T\int_\O \l| \ve\p_t Q_\ve  -\nabla_q d^F_\ve(Q_\ve)   \div \xi \r|^2\, dxdt \nonumber  \\
& +\frac 1{\ve^2}\int_0^T\int_\O \(\ve^2 \l| \p_t Q_\ve  \r|^2-|\HH_\ve|^2+  \l| \HH_\ve-\ve \HH_I |\nabla Q_\ve|\r|^2\)\, dxdt\nonumber\\
&\qquad \leq \frac{e^{(1+T)C(I_0)}}\ve E_\ve[Q_\ve | I](0) \leq e^{(1+T)C(I_0)}.  \end{align}
This together with the   orthogonal projection \eqref{projection1} yields
\begin{align}\label{energy2}
 \frac 1{\ve^2}\int_0^T\int_\O \l| \ve\p_t Q_\ve  -\ve\Pi_{Q_\ve} \p_t Q_\ve \r|^2+\frac 1{\ve^2}\int_0^T\int_\O \l| \ve\Pi_{Q_\ve} \p_t Q_\ve  -\nabla_q d^F_\ve(Q_\ve)   \div \xi \r|^2  \leq e^{(1+T)C(I_0)}.   \end{align}
The above two estimates together  with  \eqref{energy bound2} implies  \eqref{energy bound4}. Moreover, \eqref{space der bound local}
 follows from \eqref{energy1} and  \eqref{energy bound3}. 
  Now we turn to the time derivative. It follows from \eqref{energy1} that
 \begin{equation}
 \frac 1{\ve^2}\int_0^T\int_\Omega \(\ve^2 \l| \p_t Q_\ve  \r|^2-|\HH_\ve|^2+  \l| \HH_\ve-\ve \HH_I |\nabla Q_\ve|\r|^2\)\, dxdt\leq e^{(1+T)C(I_0)}.
 \end{equation}
Using \eqref{mean curvature app}, we expand the integrand  in the above estimate and deduce
\begin{equation}\label{time est1}
\int_0^T\int_\Omega |\p_t Q_\ve+\HH_I \cdot\nabla Q_\ve|^2\, dx dt\leq e^{(1+T)C(I_0)}.
\end{equation}
So 
 combining \eqref{space der bound local} with \eqref{time est1}   leads us to \eqref{time der bound local}.
\end{proof}

 With the above uniform  estimates, we can prove the following convergence result.
\begin{prop}
There exists a subsequence of $\ve_k>0$  such that
\begin{subequations}\label{global control}
\begin{align}\label{global control1}
\left[\p_t Q_{\ve_k},Q_{\ve_k} \right]=\left[\p_t Q_{\ve_k}-\Pi_{Q_{\ve_k}}\p_t Q_{\ve_k},Q_{\ve_k} \right]\xrightarrow{k\to \infty}& \bar{S}_0(x,t)~\text{weakly in}~  L^2(0,T;L^2(\Omega)),\\
\left[\p_i Q_{\ve_k},Q_{\ve_k} \right]=\left[\p_i Q_{\ve_k}-\Pi_{Q_{\ve_k}}\p_i Q_{\ve_k},Q_{\ve_k} \right]\xrightarrow{k\to \infty}& \bar{S}_i(x,t)~\text{weakly-star in}~  L^\infty(0,T;L^2(\Omega))
\label{global control t}
\end{align}
\end{subequations}
for  $1\leq i\leq d$. Moreover, 
\begin{subequations}\label{weak strong convergence}
\begin{align}
\p_t Q_{\ve_k}\xrightarrow{ k\to\infty } \p_t Q&,~\text{weakly in}~  L^2(0,T;L^2_{loc}(\Omega^\pm(t))),\label{deri con2}\\
\nabla Q_{\ve_k}\xrightarrow{k\to\infty }  \nabla Q&,~\text{weakly in}~  L^\infty(0,T;L^2_{loc}(\Omega^\pm(t))),\label{deri con}\\
  Q_{\ve_k}\xrightarrow{k\to\infty }    Q & ,~\text{strongly in}~  C([0,T];L^2_{loc}(\Omega^\pm(t))),\label{deri con1}
\end{align}
\end{subequations}
where  $Q=Q(x,t)$ is represented as  
\begin{align}\label{limuni}
&Q (x,t)=s^\pm \(\mathrm{u}(x,t) \otimes \mathrm{u}(x,t)-\frac 13I_3\)~\text{a.e.}~(x,t)\in \Omega^\pm_T\end{align}
for some unit vector field
\begin{equation}\label{orientation integrable}
\mathrm{u}\in L^\infty(0,T;H^1(\Omega^+(t);\BS))\cap H^1(0,T;L^2(\Omega^+(t);\BS))\cap C([0,T];L^2(\Omega^+(t);\BS)).
\end{equation}
\end{prop}
\begin{proof}
We first deduce from  \eqref{bulkED} and \eqref{quasidistance}    that $d^F(Q)$ is an isotropic function, which only depends on the eigenvalue of $Q\in \Q$. So by \eqref{psi}, the mollified distance function  $d^F_\ve(Q)$ is isotropic and smooth in $Q$. By \cite{MR757959} there exists a smooth symmetric function $g(\lambda_1,\lambda_2,\lambda_3)$ such that $d^F_\ve(Q)=g(\lambda_1(
Q),\lambda_2(Q),\lambda_3(Q))$.  Let $Q_0\in \Q$ be a matrix having   distinct eigenvalues, then $\lambda_i(Q)$ as well as the eigenvectors $\nn_i(Q)$ are  real-analytic  functions of $Q$ near $Q_0$, and then  by chain rule
\begin{equation}
\frac{\p d^F_\ve(Q)}{\p Q}=\sum_{k=1}^3\frac{\p g}{\p \lambda_k }\frac {\p \lambda_k }{\p Q}=\sum_{k=1}^3\frac{\p g}{\p \lambda_k }\nn_k(Q)\otimes\nn_k(Q),~\text{in a neighborhood of}~Q_0.
\end{equation}
In a neighborhood of $Q_0$, we also have $Q=\sum_{k=1}^3\lambda_k(Q)\nn_k(Q)\otimes \nn_k(Q)$. So we have
\begin{equation}
\left[\nabla_q d^F_\ve(Q), Q\right]=0,
\end{equation}
holds in a neighborhood of  $Q_0$  having distinct eigenvalues, and thus for every $Q\in \Q$ by continuity.
Now in view of \eqref{projection1}, we have
\begin{equation}
 [\Pi_{Q_\ve}\p_t Q_\ve(x,t),Q_\ve(x,t)]=0,\quad  [\Pi_{Q_\ve}\p_i Q_\ve(x,t),Q_\ve(x,t)]=0~a.e. ~(x,t)\in \O_T
\end{equation}
for $1\leq i\leq d$.
This together with   \eqref{L infinity bound1} and \eqref{energy bound4}
implies
\begin{align}
&\|\left[\p_t Q_\ve,Q_\ve \right]\|_{L^2(0,T;L^2(\O))}+\|\left[\nabla  Q_\ve,Q_\ve \right]\|_{L^\infty(0,T;L^2(\O))}\nonumber\\
=&\|\left[\p_t Q_\ve-\Pi_{Q_\ve}\p_t Q_\ve,Q_\ve \right]\|_{L^2(0,T;L^2(\O))}+\|\left[\nabla  Q_\ve-\Pi_{Q_\ve}\nabla Q_\ve,Q_\ve \right]\|_{L^\infty(0,T;L^2(\O))}\leq C
\end{align}
for some $C$ independent of $\ve$. Combining this estimate with weak compactness implies   \eqref{global control}.

 It follows from \eqref{space der bound local}, \eqref{time der bound local}, \eqref{L infinity bound1}, and the Aubin-Lions lemma  that,  for any $\delta>0$,   there exists a subsequence $\ve_k=\ve_k(\delta)>0$ such that
\begin{subequations}
\begin{align}
\p_t Q_{\ve_k}\xrightarrow{ k\to\infty } \p_t \bar{Q}_{\delta}&,~\text{weakly in}~  L^2(0,T;L^2(\Omega^\pm(t)\backslash I_t(\delta))),\\
\nabla Q_{\ve_k}\xrightarrow{k\to\infty } \nabla \bar{Q}_{\delta}&,~\text{weakly-star in}~  L^\infty(0,T;L^2(\Omega^\pm(t)\backslash I_t(\delta))),\\
   Q_{\ve_k}\xrightarrow{k\to\infty }    \bar{Q}_{\delta}&,~\text{weakly-star in}~  L^\infty(\Omega\times (0,T)),\\
     Q_{\ve_k}\xrightarrow{k\to\infty }  \bar{Q}_{\delta}&,~\text{strongly in}~  C([0,T];L^2(\Omega^\pm(t)\backslash I_t(\delta))).
  \end{align}
\end{subequations}
By a diagonal argument, we infer there exists
\begin{equation}\label{regular limit}
Q\in L^2(0,T;H^1_{loc}(\Omega^\pm(t)))\cap L^\infty(\O^\pm), ~\text{with}~\p_t Q\in L^2(0,T;L^2_{loc}(\Omega^\pm(t)))
\end{equation}
such that \eqref{weak strong convergence} holds.
Moreover,  for almost every $t\in [0,T]$ and every $\delta>0$, there holds
\[Q(x,t)=\bar{Q}_\delta(x,t)  ~\text{for a.e.}~t\in (0,T),~x\in \Omega^\pm(t)\backslash I_t(\delta).\]
To prove \eqref{limuni}, using  \eqref{deri con1}, \eqref{space der bound local}, and  Fatou's lemma, we deduce that 
\begin{equation}
  F(Q)=0~\text{a.e. in } (x,t)\in  \O^+_T.\end{equation}
So we deduce from  \eqref{lemma1} that
\begin{align}
Q(x,t)=0~\text{ a.e. in }  \Omega^-_T,\qquad Q(x,t)\in \N~\text{ a.e. in }  \Omega^+_T.
\end{align}
This together with \eqref{regular limit} and the orientability theorem by
Ball--Zarnescu \cite[Section 3.2]{MR2847533} implies    that
$Q$ is uniaxial \eqref{limuni}
for some 
\begin{equation}
\mathrm{u}\in L^\infty(0,T;H^1_{loc}(\Omega^+(t);\BS)) \text{  with  }\p_t \mathrm{u} \in L^2(0,T;L^2_{loc}(\Omega^+(t);\BS)).\label{orientation}
\end{equation}
It remains to improve the integrability of $\nabla_{x,t} \mathrm{u}$.  To this end, we choose a sequence
\begin{equation}\label{cut-off}
\psi_\ell(x,t)\in C_c^\infty(\Omega^+_T)~\text{ such that }~\psi_\ell(x,t)\xrightarrow{\ell\to\infty} \mathbf{1}_{\Omega^+_T}(x,t).
\end{equation}
It follows from  \eqref{global control1}, \eqref{global control t} and \eqref{weak strong convergence} that for almost every $(x,t)\in \O^+_T$, there holds
\begin{equation}\label{precise S}
\psi_\ell\bar{S}_i=\psi_\ell \left[\p_i Q,Q \right] =\psi_\ell\(\p_i \mathrm{u}\otimes \mathrm{u}-\mathrm{u}\otimes \p_i \mathrm{u}\)=\psi_\ell (\p_i \mathrm{u}\wedge \mathrm{u}).
\end{equation}
Since $\bar{S}_i$ are $L^2$ integrable in $\O_T$, sending $\ell\to \infty$ and applying the dominated convergence theorem to the above identity  lead us to
\begin{subequations}
\begin{align}
&\p_t \mathrm{u}\wedge \mathrm{u}\in L^\infty(0,T;L^2(\Omega^+(t))),\\
&\p_i \mathrm{u}\wedge \mathrm{u}\in L^2(0,T;L^2(\Omega^+(t))), ~\text{ for } i\in\{1,\cdots,d\}.
\end{align}
\end{subequations}
Retaining that  $\mathrm{u}$ maps into $\BS$, we deduce 
\[|\p_t \mathrm{u}|^2 =|\p_t \mathrm{u}\wedge \mathrm{u}|^2,\qquad |\p_i \mathrm{u}|^2=|\p_i \mathrm{u}\wedge \mathrm{u}|^2~a.e.~\text{in}~\O^+_T,~1\leq i\leq d.\]
So we improve \eqref{orientation} to \eqref{orientation integrable}.
\end{proof}
  
 \begin{proof}[Proof of Theorem \ref{main thm}]
 We associate each testing vector field  $\varphi(x,t)=(\varphi_1,\varphi_2,\varphi_3)\in C^1(\overline{\O_T},\R^3)$ a matrix-value function by
\begin{equation}\label{wedge matrix}
\Phi(x,t)=\begin{pmatrix}
0&-\varphi_3 &\varphi_2\\
\varphi_3 & 0 & -\varphi_1\\
-\varphi_2 & \varphi_1 & 0
\end{pmatrix}
\end{equation}
Since $[\nabla_q F(Q_{\ve_k}), Q_{\ve_k}]=0$,    applying the anti-symmetric product   $[\cdot, Q_{\ve_k}] $ to    \eqref{Ginzburg-Landau} and integration by parts over $\O_T$  yields
\begin{align}
\int_{\Omega_T}\left[\p_t Q_{\ve_k},  Q_{\ve_k}\right] :\Phi  \, dxdt + \int_{\Omega_T}  \sum_{j=1}^3[\p_j Q_{\ve_k},  Q_{\ve_k}]:\p_j \Phi \, dxdt =0.
\end{align}
Note that no boundary integral will occur  due to \eqref{bc of omega}.  Recall that we denote $I_t(\delta)$ the $\delta-$ neighborhood of $I_t$.
Equivalently, we can write the above equation by

\begin{align}
& \sum_{\pm}\int_0^T\int_{\Omega^\pm(t)\backslash I_t(\delta)}\(\left[\p_t Q_{\ve_k},  Q_{\ve_k}\right] :\Phi  +  \sum_{j=1}^3[\p_j Q_{\ve_k},  Q_{\ve_k}]:\p_j \Phi \, \)\, dx dt\nonumber\\
 &+\int_0^T\int_{  I_t(\delta)}\(\left[\p_t Q_{\ve_k},  Q_{\ve_k}\right] :\Phi  +  \sum_{j=1}^3[\p_j Q_{\ve_k},  Q_{\ve_k}]:\p_j \Phi \, \)\, dxdt=0.
\end{align}
Using \eqref{weak strong convergence}, \eqref{global control} and \eqref{limuni}, we can pass $k\to \infty$ and yield
\begin{align}
&\int_0^T\int_{\Omega^+(t)\backslash I_t(\delta)}\(\left[\p_t Q,  Q\right] :\Phi  +  \sum_{j=1}^3[\p_j Q,  Q]:\p_j \Phi\)\, dxdt\nonumber\\
&\qquad +\int_0^T\int_{  I_t(\delta)}\(\bar{S}_0 :\Phi  + \sum_{j=1}^3 \bar{S}_j:\p_j \Phi \)\, dxdt=0.
\end{align}
Substituting \eqref{limuni}  and \eqref{wedge matrix} into the above identity yield
\begin{align}
& \int_0^T\int_{\Omega^+(t)\backslash I_t(\delta)}\(\p_t \mathrm{u}\wedge \mathrm{u}\cdot \varphi +\sum_{j=1}^3(\p_j  \mathrm{u}\wedge \mathrm{u})\cdot \p_j \varphi\)\, dxdt\nonumber\\&\qquad   + \int_0^T\int_{  I_t(\delta)}\(\bar{S}_0 :\Phi  +  \sum_{j=1}^3\bar{S}_j:\p_j \Phi \)\, dxdt=0.\end{align}
Due to \eqref{orientation integrable} we have the absolute continuity of $\p_t \mathrm{u}\wedge \mathrm{u}$ and $\nabla \mathrm{u}\wedge \mathrm{u}$ in $\O^+_T$. Moreover,  \eqref{global control} implies    the absolute continuity of $\{\bar{S}_i\}_{0\leq i\leq d}$ in $\O_T$. So   we can pass to the limit  $\delta\to 0$ in the above identity, which  yields
 \begin{equation}
 \int_0^T\int_{\Omega^+(t)}\p_t \mathrm{u}\wedge \mathrm{u}\cdot \varphi\, dxdt +\int_0^T\int_{\Omega^+(t)}\sum_{j=1}^3(\p_j  \mathrm{u}\wedge \mathrm{u})\cdot \p_j \varphi\, dxdt   =0.\end{equation}
 This concludes  the proof of Theorem \ref{main thm}.
 \end{proof}

 \ \
 \noindent{\it Acknowledgements}.   T.Laux is funded by the Deutsche Forschungsgemeinschaft (DFG, German Research Foundation) under Germany's Excellence Strategy -- EXC-2047/1 -- 390685813.  Y. Liu is partially supported by NSF of China under Grant  11971314.



\end{document}